\documentclass[11pt]{article}

\usepackage{amsmath,amssymb,amsthm,amscd,eucal}
\usepackage{latexsym}
\usepackage{color}
\usepackage{enumerate}

\newtheorem{thm}[equation]{Theorem}
\newtheorem{cor}[equation]{Corollary}
\newtheorem{lemma}[equation]{Lemma}
\newtheorem{prop}[equation]{Proposition}

\newtheorem{conv}[equation]{Conventions}
\newtheorem*{conj*}{Conjecture}
\theoremstyle{definition}

\newtheorem{remark}[equation]{Remark}
\newtheorem{exam}[equation]{Example}

\numberwithin{equation}{section}

\newcommand{\fa} {\mathfrak A}
\newcommand{\zfa}{\mathsf{Z}(\fa)}

\newcommand{\FF}{\mathbb{F}}  
  
\newcommand{\ZZ}{\mathbb{Z}}

\newcommand{\A}{\mathsf{A}} 
  
\newcommand{\BBg}{\mathsf{B}_g}  
\newcommand{\BBh}{\mathsf{B}_h}  
\newcommand{\DD}{\mathsf{R}}

\newcommand{\G}{\mathbb{G}}

\newcommand{\PP}{{\mathbb P}}
\newcommand{\pr}{{\mathsf u}}
\newcommand{\PR}{\mathsf{P}}

\newcommand{\ve}{\varepsilon}
\newcommand{\vv}{\mathsf{q}}

\newcommand\chara{\mathsf {char}}
\def\inder{\mathsf {Inder_\FF}}
\def\der{\mathsf {Der_\FF}}

\def\hoch{\mathsf{HH^1}} 

\def\ad{\mathsf {ad}}
\def\aut{\mathsf {Aut_\FF}}
\def\im{\mathsf{im}}

\def\degg{\mathsf{deg\,}}
\def\cent1{\mathsf{Z}(\A_1)}
\def\centh{\mathsf{Z}(\A_h)}
\def\half{\frac{1}{2}}

\newcommand{\modd}{\mathsf{\, mod\,}}  

\usepackage[letterpaper,margin=1.48in]{geometry}

\begin{document}
\title{A Parametric Family of Subalgebras of the Weyl Algebra \\ 
I. Structure and Automorphisms}
\author{Georgia Benkart, Samuel A.\ Lopes\thanks{Research funded by the European Regional Development Fund through the programme COMPETE and by the Portuguese Government through the FCT -- Funda\c c\~ao para a Ci\^encia e a Tecnologia under the project PEst-C/MAT/UI0144/2011.}, and Matthew Ondrus}
\date{}
\maketitle  
\vspace{-.25 truein}  
\begin{abstract}{An Ore extension over a polynomial algebra $\FF[x]$ is either a quantum plane, a quantum Weyl algebra, or an infinite-dimensional unital associative algebra $\A_h$ generated by elements $x,y$,  which satisfy $yx-xy = h$, where $h\in \FF[x]$. We investigate the family of algebras $\A_h$ as $h$ ranges over all the polynomials in $\FF[x]$.  When $h \neq 0$,  the algebras $\A_h$ are subalgebras of the Weyl algebra $\A_1$ and can be viewed as differential operators with polynomial coefficients.    We give an exact description of the automorphisms of $\A_h$ over arbitrary fields $\FF$ and describe the invariants in $\A_h$ under the automorphisms.   We determine the  center,  normal elements,  and height one prime ideals of $\A_h$, localizations and Ore sets for $\A_h$, and the Lie ideal  $[\A_h,\A_h]$.  We also  show that $\A_h$ cannot be realized as a generalized Weyl algebra over $\FF[x]$,  except when $h \in \FF$.   In two sequels to this work, we completely describe the irreducible modules and derivations of $\A_h$ over any field.} \end{abstract}

\begin{section}{Introduction} \end{section}
The focus of this paper is on a family of infinite-dimensional unital associative algebras $\A_h$ parametrized by a polynomial $h = h(x) \in \FF[x]$, where $\FF$ is an arbitrary field.   The algebra $\A_h$ has generators $x,y$,  which satisfy the defining relation $yx = xy + h$,  or equivalently,  $[y,x] = h$, where $[y,x] = yx-xy$.   The Ore extensions whose underlying ring is $\FF[x]$ fall into three specific types.  They are quantum planes, quantum Weyl algebras, or one of the algebras $\A_h$ (compare Lemma \ref{lem:poly} below).  Quantum planes and quantum Weyl algebras are examples of generalized Weyl algebras, and as such, have been studied extensively.   It is the aim of our work to investigate the family of algebras $\A_h$ as $h$ ranges over all the polynomials in $\FF[x]$.   The algebras $\A_h$  are left and right Noetherian domains.  As modules over $\FF[x]$,  they are free with basis $\{y^n \mid n \in \ZZ_{\geq 0}\}$.   Each algebra $\A_h$ with $h \neq 0$ can be viewed as a subalgebra of the Weyl algebra $\A_1$ 
and thus has a representation as differential operators on $\FF[x]$, where $x$ acts by multiplication and $y$
by $h \frac{d}{dx}$,  so that $[h \frac{d}{dx}, x] = h$ holds. 

There are several widely-studied examples of algebras in this family.   The algebra $\A_0$
is the polynomial algebra $\FF[x,y]$;   $\A_1$ is the Weyl algebra; and  $\A_x$ is the universal
enveloping algebra of the two-dimensional non-abelian Lie algebra (there is only one such
Lie algebra up to isomorphism).  The algebra $\A_{x^2}$ is often referred to as the Jordan
plane.  It arises in noncommutative algebraic geometry (see for example, \cite{SZ94} and \cite{AS95})  and
exhibits many interesting features such as being Artin-Schelter regular of dimension 2. 
In a series of articles \cite{shirikov05}--\cite{shirikov07-2},  Shirikov
has undertaken an extensive study of the automorphisms, derivations, prime ideals, and modules
of the algebra $\A_{x^2}$.  These investigations have been extended by  Iyudu  in 
recent work  \cite{I12}  to include results on varieties of finite-dimensional modules of
$\A_{x^2}$ over algebraically closed fields of characteristic zero.  Cibils, Lauve, and Witherspoon \cite{CibLauWit09}  have used quotients of the algebra $\A_{x^2}$ 
and cyclic subgroups of their automorphism groups to
construct new examples of finite-dimensional Hopf algebras in prime characteristic
which are Nichols algebras.

There are striking similarities in the behavior of the algebras $\A_h$ as $h$ ranges over the polynomials
in $\FF[x]$.  For that reason, we believe that studying them as one family provides much insight into their
structure, derivations, automorphisms, and modules.    
In this paper, we determine the following:
\medskip 

\begin{itemize}
\item embeddings of $\A_g$ into $\A_f$ \quad  (Section 3)
\item localizations and Ore sets for $\A_h$  \quad  (Section 4)
\item the center of $\A_h$ \quad (Section 5) 
\item the Lie ideal  $[\A_h,\A_h]$ of $\A_h$   \quad  (Section 6)
\item the normal elements and the prime ideals of $\A_h$  \ \  (Section 7)
\item the automorphism group $\fa = \aut(\A_h)$ and its center, and the subalgebra $\A_h^{\fa}$ of $\fa$-invariants in $\A_h$  \quad  (Section 8)
\item the relationship of $\A_h$ to generalized Weyl algebras \quad  (Section 9).
\end{itemize} 
 In the sequel  \cite{BLO2}, we determine the irreducible modules and the primitive ideals of $\A_h$  in arbitrary characteristic and construct indecomposable $\A_h$-modules of arbitrarily large dimension. 
In further work \cite{BLO3}, 
we completely describe  the Lie algebra $\der(\A_h)$ of $\FF$-linear derivations and the first Hochschild cohomology
$\hoch(\A_h) = \der(\A_h)/\inder(\A_h)$ of $\A_h$ over arbitrary fields $\FF$.   Our investigations extend
earlier results of Nowicki \cite{Now04}.  In particular, 
we determine  the Lie bracket in $\hoch (\A_h) := \der (\A_h) / \inder (\A_h)$,  construct a maximal nilpotent ideal of $\hoch (\A_h)$,  and explicitly describe the structure of the corresponding quotient in terms of the Witt algebra  (centreless Virasoro algebra) of vector fields on the unit circle when $\chara(\FF) = 0$.
 
\begin{section}{Ore Extensions} \end{section}
\begin{subsection}{Generalities}\end{subsection}
An Ore extension $\A = \DD[y,\sigma, \delta]$  is built from a unital associative (not necessarily commutative) algebra $\DD$ over a field $\FF$, an $\FF$-algebra endomorphism $\sigma$ of $\DD$, and a $\sigma$-derivation of $\DD$, where by a  $\sigma$-derivation $\delta$,  we mean that  $\delta$ is $\FF$-linear and $\delta (rs) = \delta(r)s + \sigma (r) \delta (s)$ holds for all $r,s \in \DD$.  Then  $\A = \DD[y, \sigma, \delta]$ is the algebra generated by $y$ over $\DD$ subject to the relation 
$$yr= \sigma(r)y + \delta(r) \qquad \hbox{\rm for all} \ r \in \DD.$$     
The endomorphisms $\sigma$ considered in this paper will be automorphisms of $\DD$.  
The following are standard facts about Ore extensions. 
\medskip

\begin{thm}\label{thm:basicFactsOre}
Let $\A = \DD[y,\sigma, \delta]$ be an Ore extension over a unital associative algebra $\DD$ over a field $\FF$ such that $\sigma$ is an automorphism.
\begin{enumerate}
\item[{\rm (1)}]  $\A$ is a free left and right $\DD$-module with basis $\{ y^n  \mid n \ge 0 \}$.
\item[{\rm (2)}]  If $\DD$ is left (resp.~right) Noetherian, then $\A$ is left (resp.~right) Noetherian.
\item[{\rm (3)}]  If $\DD$ is a domain, then $\A$ is a domain.
\item[{\rm (4)}] The units of $\A$ are the units of $\DD$.
\end{enumerate}
\end{thm}

\begin{subsection}{Ore Extensions with Polynomial Coefficients} \end{subsection}

We are concerned with Ore extensions $\A = \DD[y,\sigma,\delta]$  with  $\DD = \FF[x]$, a polynomial algebra in the indeterminate $x$,  and $\sigma$ an automorphism of $\DD$.  In this case,  $\sigma$ has the form $\sigma(x) = \alpha x+ \beta$ for some $\alpha, \beta \in \FF$ with $\alpha \neq 0$.    Hence,  $\A$ is isomorphic to the unital associative algebra over $\FF$ with generators $x,y$ subject to the defining relation $y x  = ( \alpha x+ \beta)y + h$, where $h$ is the polynomial given by $h(x) = \delta(x)$.    The next lemma reduces the study of  such Ore extensions to three specific types of algebras.   This result is essentially contained in Observation 2.1 of the paper \cite{AVV87} by Awami, Van den Bergh, and Van Oystaeyen (compare also \cite[Prop.~3.2]{AD97}), although the division into cases here is somewhat different from that given in those papers. 

\medskip
\begin{lemma}\label{lem:poly}  Assume $\A = \DD[y,\sigma,\delta]$ is an Ore extension with $\DD = \FF[x]$,
a polynomial algebra over a field $\FF$ of arbitrary characteristic, and $\sigma$ an automorphism of $\DD$.   Then $\A$ is isomorphic
to one of the following:
\begin{itemize}
\item[{\rm (a)}]   a  quantum plane
\item[{\rm (b)}]   a quantum Weyl algebra
\item[{\rm (c)}]   a unital associative algebra $\A_h$  with generators $x,y$ and defining relation
$yx = xy + h$ for some polynomial $h = h(x) \in \FF[x]$.  
\end{itemize}
\end{lemma}
 
Quantum planes and quantum Weyl algebras are generalized Weyl algebras in the sense of \cite[1.1]{bavula:gwar93} and their structure and  irreducible modules have been studied extensively in that context.  

Our aim in this paper is to give a detailed investigation of the algebras that arise in (c) of Lemma \ref{lem:poly}.   The algebra $\A_h$ is the Ore extension $\DD[y,\mathsf{id}_{\DD}, \delta]$ obtained from
the polynomial algebra $\DD = \FF[x]$ over the field $\FF$ by taking $h \in \DD$,  $\sigma$ to be the identity automorphism $\mathsf{id}_\DD$  on $\DD$, and $\delta: \DD \rightarrow \DD$ to be  the $\FF$-linear derivation with $\delta(f) = f'h$ for all $f \in \DD$,  where $f'$ denotes  the usual derivative of $f$ with respect to $x$.    
 
It  is convenient to regard $\A_h$ as  the unital associative algebra over $\FF$ with generators $x$, $y$ and defining relation $[y,x] = h$.  
Then  $[y,f] = \delta(f) = f'h$ holds in $\A_h$ for all $f \in \DD$.   Theorem \ref{thm:basicFactsOre}
 implies that $\A_h$ is both a left and right Noetherian domain with units $\FF^* 1$ and that 
$$\A_h = \bigoplus_{i \ge 0} \DD y^i,$$
where $\DD = \FF[x]$.     Hence,  $\{x^j y^i \mid j,i  \in \mathbb Z_{\ge 0}\}$ is a basis for $\A_h$ over $\FF$,   and  $\A_h$ has
Gelfand-Kirillov (GK) dimension 2 by \cite[Cor.\ 8.2.11]{McR01}.  
 
\begin{section}{The Embeddings $\A_g \subseteq \A_f$} \end{section}

Fix nonzero $f,g \in \DD = \FF[x]$.  In order to distinguish generators for the algebras
$\A_f$ and $\A_g$,   we will assume those for $\A_f$ are $x, y,1$, and those for $\A_g$ are $x, \tilde y,1$.  

\medskip
\begin{lemma} \label{lem:embed}  For $f, g, \in \DD$, suppose that $f\, |\, g$ and $g = fr$.   Then the map   $\psi:  \A_g \rightarrow \A_f$ with 
$$x \mapsto x, \hspace{.5in} \tilde y \mapsto  yr$$
gives  an embedding of $\A_g$ into $\A_f$.   \end{lemma} 
\begin{proof}
This follows directly  from the observation that   
$[yr, x] =  [y, x]r=  fr = g$.
\end{proof}
 
\begin{cor} \label{cor:embed} 
For all nonzero $h  \in \FF[x]$,   there is an embedding of the algebra $\A_h$  into the Weyl algebra  $\A_1$.  \end{cor}
\medskip
 
Because we often use the embedding in Corollary \ref{cor:embed} as a mechanism for proving results,
and because the structure of $\A_0 = \FF[x,y]$ is very well understood, for the remainder of this paper
 we adopt the following conventions:
 \begin{conv}  \label{con:gens}  \qquad 
 
 \begin{itemize}
 \item  $\DD = \FF[x]$,  and the polynomial $h \in \DD$ is nonzero;
 \item   the generators of the Weyl algebra $\A_1$ are  $x, \, y, \, 1$;
 \item  the generators of the algebra $\A_h$ are $x, \, \hat y,\, 1$;    \label{eqn:def_yHat}
 \item  when $\A_h$ is viewed as a subalgebra of $\A_1$, then $\hat y = yh$.  
 \end{itemize}
 \end{conv}

The following result provides an important tool for recognizing elements of $\A_h$ inside of $\A_1$.
\medskip \begin{lemma}\label{lem:A_h-insideA_1}
Regard $\A_h \subseteq \A_1$ as in Conventions \ref{con:gens}.  Then 
$$\A_h = \bigoplus_{i \ge 0} \DD h^iy^i = \bigoplus_{i \ge 0} y^i h^i \DD.$$
\end{lemma}
\begin{proof}

We show that  $\bigoplus_{i=0}^n \hat y^i \DD = \bigoplus_{i = 0}^n  y^ih^i \DD$ for all $n \geq 0$, and from that
we can immediately conclude $\A_h = \bigoplus_{i \geq 0} y^i h^i \DD$.  
Observe for $j \in \ZZ,$
\begin{equation}\label{eqn:(yHat)(h)}
( \hat y+jh')h = h ( \hat y + (j+1)h').
\end{equation}
Also note that $yh = \hat y$ and $y^2 h^2 = y \hat y h = yh ( \hat y + h') = \hat y ( \hat y + h')$ hold.   It follows easily from \eqref{eqn:(yHat)(h)} and induction that 
\begin{equation}\label{eqn:y^kh^k}
y^i h^i = \hat y ( \hat y+h')( \hat y+2h') \cdots ( \hat y + (i-1)h')  \in \A_h.\end{equation}
This implies that   $ y^ih^i \DD  \subseteq \bigoplus_{j=0}^n \hat y^j \DD$ for $0 \leq i \leq n$. 
For the other containment,  we argue that $\hat y^n  \in \bigoplus_{i \ge 0}^n  y^ih^i \DD$ 
by induction on $n$,  with the $n = 1$ case simply being the definition,  $\hat y = yh$.  
Now from  (\ref{eqn:y^kh^k}) with $i = n$,   we have that
$y^nh^n =  \hat y^n + a$,  where $a \in  \sum_{j=0}^{n-1} \hat y^j \DD$.    Thus by induction,
$\hat y^n  =  y^n h^n - a $ where $a \in  \bigoplus_{i = 0}^{n-1}  y^ih^i \DD$, and 
the containment $\bigoplus_{i=0}^n \hat y^i \DD \subseteq \bigoplus_{i = 0}^n  y^ih^i \DD$  holds.

The anti-automorphism of $\A_1$ with $x \mapsto x$ and $y \mapsto -y$ sends $\hat y$ to $-\hat y + h'$.  Hence,
it restricts to an anti-automorphism of $\A_h$. When applied to $\A_h = \bigoplus_{i \geq 0} y^i h^i \DD$,
it gives $\A_h = \bigoplus_{i \geq 0} \DD h^i y^i$ and shows that 
\begin{equation}\label{eqn:h^ky^k}
h^iy^i = (\hat y -ih')( \hat y-(i-1)h') \cdots ( \hat y -h')  \in \A_h.\end{equation}  \end{proof}

\begin{section}{Localizations  and Ore Sets} \end{section}

The embedding $\A_{h}\subseteq\A_{1}$  suggests that there is a strong relationship between the skew fields of fractions of $\A_{h}$ and $\A_{1}$. In this section, we will see that  in fact these skew fields are identical. To show this result, we describe certain Ore sets in $\A_1$ and $\A_h$.   Our starting point is a computational lemma.  

\bigskip
\begin{lemma}\label{lem:yg^m_in_g^{m-1}A}
Fix $f, h \in \DD$, with $f\neq 0$.  If $0 \leq j \leq m$, then $\hat y^j f^m \in f^{m-j} \A_h$.
\end{lemma}
\begin{proof}
Observe that 
$$\hat yf^m = f^m\hat y + (f^m)' h   \in f^{m-1} \A_h.$$
Repeated application of this gives the claim.
\end{proof}

\begin{lemma}\label{lem:OreSetGeneral}
Fix $f, h \in \DD$, with $f\neq 0$.  Then the set $\Sigma = \{ f^n \mid n\geq 0 \}$ is a left and right Ore set of regular elements in $\A_h$.
\end{lemma}
\begin{proof}
That $\Sigma$ consists of regular elements follows from the fact that $\A_h$ is a domain.  Let $a \in \A_h$ and $s \in \Sigma$.  We must show that there exist $a_1 \in \A_h$ and $s_1 \in \Sigma$ such that $as_1 = sa_1$.  It is enough to consider the case $s = f$. Write $a = \sum_{i=0}^k r_i \hat y^i$ and set $s_1 = f^{k+1}$.  By Lemma \ref{lem:yg^m_in_g^{m-1}A}, we see that 
$$as_1 = \sum_{i=0}^k r_i \hat y^i f^{k+1} \in \sum_{i=0}^k r_i f \A_h \subseteq f \A_h = s \A_h.$$
A similar argument shows that $\Sigma$ is a left Ore set.
\end{proof}
 
\begin{cor}\label{cor:hPowers-OreSet}
Regard $\A_h$ as a subalgebra of $\A_1$ as in Conventions \ref{con:gens}. Let $\Sigma=\{ h^{n}\mid n\geq 0 \}$. Then $\Sigma$ is a left and right Ore set of regular elements in both $\A_{1}$ and $\A_{h}$, and the corresponding localizations are equal:
$$\A_{1}\Sigma^{-1}=\A_{h}\Sigma^{-1}.$$
\end{cor}
\begin{proof}
By applying Lemma \ref{lem:OreSetGeneral} to $\A_{1}$ with $\Sigma=\{ h^{n}\mid n\geq 0 \}$, and then to $\A_{h}$ with $f = h$, we see that $\Sigma$ is a left and right Ore set in both $\A_1$ and $\A_h$.  Clearly $\A_h \Sigma^{-1} \subseteq \A_1 \Sigma^{-1}$ since $\A_h \subseteq \A_1$.  That $\A_1 \Sigma^{-1} \subseteq \A_h \Sigma^{-1}$ follows from the fact that $\A_h \Sigma^{-1}$ contains the element $\hat y h^{-1} = yhh^{-1} = y$.
\end{proof}

\begin{cor}\label{C:Weylfield} The skew field of fractions of $\A_h$ is isomorphic to the skew field of fractions of the Weyl algebra $\A_1$
(commonly referred to as the Weyl field).  \end{cor}
 
\begin{cor} Assume $\A_h \subseteq \A_1$ as in Conventions \ref{con:gens}.    Then the following are equivalent:
\begin{itemize}
\item[{\rm (1)}]  $h \in \FF^*$.
\item[{\rm (2)}] $\A_1$ is a Noetherian (left or right) $\A_h$-module.
\item[{\rm (3)}] $\A_{1}$ is a free (left or right) $\A_{h}$-module.
\end{itemize}
\end{cor}

\begin{proof}
If $h\in\FF^{*}$,  then the embedding $\A_{h}\subseteq \A_{1}$ considered in this section is an equality.  Thus as an $\A_h$-module,  $\A_{1}$ is free of rank one,  and it is Noetherian. 

Now assume $h\notin\FF$. For each $k\geq 0$, consider the right $\A_h$-submodule 
$$
\mathcal{Y}_{k}=\A_{h}+y\A_{h}+\cdots +y^{k}\A_{h}\subseteq \A_{1}.
$$
If $\sum_{i\geq 0}r_{i}y^{i}\in\mathcal{Y}_{k}$, with $r_{i}\in\DD$, it is easy to conclude that $h$ divides $r_{i}$ for all $i\geq k+1$. Thus, $y^{k+1}\in\mathcal{Y}_{k+1}\setminus\mathcal{Y}_{k}$ and the chain of submodules
$$
(0)\subset \A_{h}=\mathcal{Y}_{0}\subset \mathcal{Y}_{1}\subset \mathcal{Y}_{2}\subset \cdots
$$
does not terminate. In particular, $\A_{1}$ is not a Noetherian $\A_{h}$-module. Since $\A_{h}$ is a Noetherian ring, it follows that $\A_{1}$ is not a finitely generated $\A_{h}$-module either. Assume there exist elements $0\neq t_{i}\in\A_{1}$, $i\in {\tt I}$, such that
\begin{equation*}
\A_{1}=\bigoplus_{i\in {\tt I}}t_{i}\A_{h}.
\end{equation*}
Consider the Ore set $\Sigma=\{ h^{n}\mid n\geq 0 \}$. It follows that $\A_{1}\Sigma^{-1}=\bigoplus_{i\in {\tt I}}t_{i}\A_{h}\Sigma^{-1}$. By Corollary~\ref{cor:hPowers-OreSet} we have $\A_{1}\Sigma^{-1}=\A_{h}\Sigma^{-1}=: \mathsf B$ and thus $\mathsf B=\bigoplus_{i\in {\tt I}}t_{i} \mathsf B$. This implies that ${\tt I}$ must be finite, as the decomposition of $1\in\mathsf B$ uses only finitely many summands. This contradicts the fact that $\A_{1}$ is not a finitely generated $\A_{h}$-module. Hence, $\A_{1}$ is not a free right $\A_{h}$-module. This proves the corollary for when $\A_{1}$ is considered as a right $\A_{h}$-module. The left-hand version is analogous.  
\end{proof}

\begin{section}{The Center of $\A_h$} \end{section}

In this section, we describe the center $\mathsf Z(\A_h)$ of $\A_h$ and show in Proposition \ref{prop:A_h-freeOverZ} that $\A_h$ is free over $\mathsf Z(\A_h)$.    In the case of the Weyl algebra,
the center is $\FF1$ when $\chara(\FF) = 0$. When $\chara(\FF) = p > 0$, the center  has been described   by Revoy in \cite{R73} (see also \cite{MakLim84}) as follows:

\medskip
\begin{lemma}\label{L:ML:centA1} 
Suppose $\chara (\FF) = p>0$.  Then the center of $\A_1$ is the unital subalgebra generated by the elements $x^p$ and $y^p$.
\end{lemma}

In determining $\centh$ for arbitrary $h$, we will use the following result which can be shown 
by a straightforward inductive argument.    
 
\begin{lemma}\label{lem:identity_y^nf}  Regard $\A_h\subseteq \A_1$  as in Conventions \ref{con:gens}.  
Let $\delta: \DD \to \DD$ be the derivation with $\delta (f) = hf'$ for all $f \in \DD$.  Then 
\begin{eqnarray}&& [\hat y^n, f] = \sum_{j=1}^n {n \choose j} \delta^j (f) \hat y^{n-j} \qquad \hbox{\rm in} \ \  \A_h, \label{eq:Ahcom} \\ 
&&[y^n,f] = \sum_{j=1}^n {n \choose j} f^{(j)} y^{n-j} \ \  \qquad \hbox{\rm in} \ \ \A_1, \label{eq:A1com} \end{eqnarray}
where $f^{(j)} = (\frac{d}{dx})^j(f)$.  
\end{lemma}

\begin{thm}\label{L:center}
Regard $\A_h\subseteq \A_1$  as in Conventions \ref{con:gens}.  
\begin{enumerate}
\item[{\rm (1)}]  If $\chara(\FF) = 0$, then the center of $\A_h$ is $\FF 1$.
\item[{\rm (2)}]  If $\chara(\FF) = p > 0$, then the center of $\A_h$ is isomorphic to the polynomial algebra $\FF[x^p, h^py^p]$, where 
\begin{equation}\label{eq:centgen} h^p y^p = y^ph^p = \hat y ( \hat y+h')( \hat y+2h') \cdots ( \hat y + (p-1)h') = \hat y^p - \frac{\delta^p(x)}{h} \hat y.\end{equation}
\end{enumerate}
\end{thm}
\begin{proof}
We first observe that $\mathsf Z(\A_1)\cap \A_h\subseteq \mathsf Z(\A_h)$, as $\A_{h}\subseteq \A_{1}$. Conversely, given $z\in \mathsf Z(\A_h)$, then $[x, z]=0$ and $0=[\hat y, z]=[yh, z]=[y, z]h+y[h, z]=[y, z]h$. Since $h\neq 0$ it follows that $[y, z]=0$ and $z\in\mathsf Z(\A_1)\cap \A_h$. Hence
\begin{equation}\label{E:cent:intersectA1}
\mathsf Z(\A_1)\cap \A_h=\mathsf Z(\A_h).
\end{equation}
If $\chara(\FF) = 0$ then $\mathsf Z(\A_h)=\FF 1$.

Now suppose that $\chara(\FF) = p>0$. Then $x^{p}, h^p y^p \in \mathsf Z(\A_1)\cap \A_{h}$. For every $k\geq 0$, $h^{kp} y^{kp} = (h^p)^k (y^p)^k = (h^py^p)^k$, thus the elements $x^p$ and $h^p y^p$ are algebraically independent, and it follows that $\FF[x^p, h^p y^p]\subseteq \mathsf Z(\A_h)$. Let $z\in \mathsf Z(\A_h)$. By (\ref{E:cent:intersectA1}), Lemma~\ref{lem:A_h-insideA_1},  and Lemma~\ref{L:ML:centA1}, we can write $z = \sum_{i \equiv 0 \modd p} r_i y^i$ with $r_{i}\in\FF[x^{p}]$ such that $h^{i}\, |\, r_{i}$ for all $i \equiv 0 \modd p$. Since $h^{i}\in\FF[x^{p}]$ for $i \equiv 0 \modd p$, there exist $c_{i}\in\FF[x^{p}]$ so that $z = \sum_{i \equiv 0 \modd p} c_i h^{i} y^i\in \FF[x^p, h^p y^p]$, and therefore $\mathsf Z(\A_h)=\FF[x^p, h^p y^p]$.    

The relation 
$h^p y^p = y^p h^p = \hat y ( \hat y+h')( \hat y+2h') \cdots ( \hat y + (p-1)h')$ is just (\ref{eqn:y^kh^k}) with $i = p$.  
To show this expression equals $\hat y^p - \frac{\delta^p(x)}{h} \hat y$,  use Lemma \ref{lem:A_h-insideA_1} to write $h^p y^p = \sum_{n=0}^p f_n \hat y^n$, where $f_n\in \FF[x]$ for all $n$ 
and $f_p = 1$.     Then
 \begin{eqnarray*}  0 &=& { [h^py^p, x]   }
 =  \sum_{n=1}^p f_n [\hat y^n, x]   
 = \sum_{n=1}^p f_n \sum_{j=1}^n {n \choose j} \delta^j(x) \hat y^{n-j} \qquad \ \ \hbox{\rm by \eqref{eq:Ahcom}} \\
   &=& f_p \delta^p(x)  +\sum_{n=1}^{p-1} f_n \sum_{j=1}^n  {n \choose j} \delta^j(x) \hat y^{n-j}   \\
 &=& \delta^p(x) + {{p-1}\choose {1}}f_{p-1} \delta(x)\hat y^{p-2} + \hbox{lower terms}.  \end{eqnarray*}
   
Since $\delta(x) = h \neq 0$, we see that $f_{p-1} = 0$.   Then the above gives 
\begin{equation*} 0 = \delta^p(x) +{ {p-2}\choose 1} f_{p-2} \delta(x) \hat y^{p-3} + \  \hbox{ lower terms}.\end{equation*}
Proceeding in this way,  we obtain $f_n = 0$ for all $n=p-1,p-2, \ldots, 2$.      As a result,  we have
$0 =  \delta^p(x) + f_1 \delta(x)$ or $f_1 =-\frac{\delta^p(x)}{h}$, since $h$ always divides $\delta^k(x)$
for $k \geq 1$.     Consequently,  $h^py^p  = \hat y^p -\frac{\delta^p(x)}{h}\hat y + f_0$.  Then
$$
0=[\hat y, \hat y^p -\textstyle{\frac{\delta^p(x)}{h}}\hat y + f_0]=[\hat y, -\textstyle{\frac{\delta^p(x)}{h}}\hat y]+[\hat y, f_0]= -[\hat y, \textstyle{\frac{\delta^p(x)}{h}}]\hat y+hf'_0,
$$
and it follows that $[\hat y, \frac{\delta^p(x)}{h}]=0$. But then  
$$\hat y^p - \hat y\textstyle{\frac{\delta^p(x)}{h}} + f_0 = \hat y^p -\frac{\delta^p(x)}{h}\hat y + f_0=h^p y^p = \hat y(\hat y+h') \cdots (\hat y + (p-1)h') \in
\hat y \A_h,$$
and hence  $f_0 \in  \hat y \A_h$.   The only way that can happen is if $f_0 = 0$ and 
$h^p y^p = \hat y^p -\textstyle{\frac{\delta^p(x)}{h}}\hat y.$ \end{proof}
 
\begin{exam} Assume $\chara(\FF) = p>0$ and  $h(x) = x^n$ for some $n \geq 1$.  Then it is easy to verify that
$$\delta^p(x) =  \left(\prod_{k=1}^{p-1}k(n-1)+1 \right) x^{np-p+1}.$$
Hence, if $n \not \equiv 1 \modd p$,  we can find $1 \leq  k < p$ with $k(n-1) \equiv -1 \modd p$ so that 
$\delta^p(x) = 0$.       This implies that when $h(x) = x^n$,    
$$\frac{\delta^p(x)}{h} = \begin{cases} 0  & \qquad  \hbox {\rm if \ \ $n \not \equiv 1 \modd p$} \\
x^{(n-1)(p-1)}  & \qquad \hbox {\rm if \ \ $n  \equiv 1 \modd p$}. \end{cases} $$
In particular,  $\centh = \FF[x^p, \hat y^p]$ whenever $h(x) = x^n$ and $n \not \equiv 1 \modd p$.   
When $n = 2$, this was shown by Shirikov in \cite{shirikov07-2}.  \end{exam}    
 
\begin{prop}\label{prop:A_h-freeOverZ}
Assume $\chara(\FF) = p>0$ and regard $\A_h  \subseteq \A_1$ as in Conventions \ref{con:gens}.  Then $\A_h$ is a free module over $\mathsf Z(\A_h)$, and the set $\{ x^i h^{j}y^{j} \mid 0 \le i, j < p \}$ is a basis.
\end{prop}
\begin{proof}
Suppose that 
\begin{equation}\label{eqn:freeBasis1}
0 = \sum_{0 \le i,j < p} c_{i,j} x^i h^j y^j,
\end{equation}
where $c_{i,j} \in \mathsf Z(\A_h) = \FF[x^p, h^py^p]$.  For $0 \le j < p$,
\begin{equation*}
\sum_{0 \le i < p} c_{i,j} x^i h^j y^j\in\bigoplus_{k\equiv j \modd p}\DD y^{k}.
\end{equation*}
Thus, (\ref{eqn:freeBasis1}) and Theorem~\ref{thm:basicFactsOre} imply that $\sum_{0 \le i < p} c_{i,j} x^i h^j y^j=0$. As $h\neq 0$, it follows that $\sum_{0 \le i < p} c_{i,j} x^i=0$ for every $0 \le j < p$.  The direct sum decomposition $\FF[x, h^p y^p]=\bigoplus_{i=0}^{p-1}\FF[x^{p}, h^p y^p]x^{i}$  then implies $c_{i,j} = 0$ for all $i, j$.

It remains to show that $\{ x^i h^{j}y^{j} \mid 0 \le i, j < p \}$ generates $\A_h$ over $\mathsf Z(\A_h)$. Let $a, b\geq 0$ and write 
\begin{equation*}
a=\tilde{a}p+i, \hspace{.5in} b=\tilde{b}p+j,
\end{equation*}
for some  nonnegative integers $\tilde{a}$, $\tilde{b}$ and  $0 \le i, j < p$. Then,
\begin{equation*}
x^a h^{b}y^{b}=\left(x^p\right)^{\tilde{a}} \left(h^{p}y^{p}\right)^{\tilde{b}}x^i h^{j}y^{j}\in \mathsf Z(\A_h) x^i h^{j}y^{j}.
\end{equation*}
As $\{ x^a h^{b}y^{b} \mid a, b\geq 0 \}$ is a basis for $\A_{h}$, by Lemma~\ref{lem:A_h-insideA_1} the result is established.  \end{proof}

\begin{remark}\hfill
\begin{enumerate}
\item[{\rm (i)}]   The algebra anti-automorphism $x\mapsto x$, $y\mapsto -y$ of $\A_{1}$ can be applied to the
basis above to show that $\{ y^{j}h^{j}x^i \mid 0 \le i, j < p \}$ is a basis for $\A_{h}$ over $\mathsf Z(\A_h)$.

\item[{\rm (ii)}]  A standard inductive argument can be used to prove that $\{ x^i y^{j}h^{j} \mid 0 \le i, j < p \}$ is also a basis for $\A_h$ over $\mathsf Z(\A_h).$
\end{enumerate}
\end{remark} 
\begin{section}{The Lie Ideal $[ \A_h, \A_h ]$} \end{section}

\begin{lemma} \label{lem:[A_h,A_h]-sub-hA_h}
Let $h \in \FF[x]$.  Then $[ \A_h, \A_h ] \subseteq h \A_h$.
\end{lemma}
\begin{proof}
Recall that $\A_h$ is spanned by elements of the form $a\hat y^\ell$ for $\ell\geq 0$ and $a\in\DD$. Thus it suffices to show that $[a \hat y^\ell , b \hat y^m] \in h\A_h$ for all $\ell, m \ge 0$ and $a, b\in\DD$. Observe that
$$
[a \hat y^\ell , b \hat y^m]=[a \hat y^\ell, b]\hat y^m+b[a \hat y^\ell, \hat y^m]=a[\hat y^\ell, b]\hat y^m-b[\hat y^m, a]\hat y^\ell,
$$ 
so it is enough to show that $[\hat y^n, f]\in h \A_h$ for all $n\geq 0$ and $f\in\DD$. This follows directly from \eqref{eq:Ahcom} as $\delta^{j}(f)\in h\DD$ for all $j\geq 1$.\end{proof}

We have the following simple description of $[\A_h, \A_h]$ for fields of characteristic 0. 

\begin{prop} \label{prop:char0-commutator}  Suppose that ${\rm char} (\FF) = 0$.  Then $h \A_h 
= [x,\A_h] = [\hat y, \A_h ] = [ \A_h, \A_h]$.
\end{prop}
\begin{proof}
By Lemma \ref{lem:[A_h,A_h]-sub-hA_h}, it suffices to prove that $h \A_h \subseteq  [ \hat y, \A_h ]$.  Note that $h \A_h = h \left(\bigoplus_{i \ge 0} \DD \hat y^i\right)$, and by the linearity of the adjoint map $\ad_{\hat y}$ (where $\ad_{\hat y}(a) = [\hat y,a]$), it is enough to show that $h g \hat y^i \in [\hat y, \A_h]$ for every $i \ge 0$ and $g \in \DD$.   Since ${\rm char}(\FF) = 0$, the element $g \in \DD$ has the form $f'$ for some $f \in \DD$, and therefore 
$$[\hat y, f\hat y^i] = [\hat y,f] \hat y^i = h f' \hat y^i = hg \hat y^i.$$ 
It remains to show that $h \A_h \subseteq  [ x, \A_h ]$. It will be more convenient to work inside  $\A_{1}$,  where  $h \A_h = h \left(\bigoplus_{i \ge 0} \DD h^i y^i\right)$. Then, for $i\geq 0$ and $g\in \DD$ we have $\frac{1}{i+1}gh^{i+1}y^{i+1}\in\A_{h}$ and 
$$
\left[ \textstyle{\frac{1}{i+1}}gh^{i+1}y^{i+1}, x\right]=\textstyle{\frac{1}{i+1}}gh^{i+1}[y^{i+1}, x]=hgh^{i}y^{i}.
$$
The linearity of $\ad_{x}$ implies that $h \A_h \subseteq  [\A_{h}, x] = [ x, \A_h ]$.
\end{proof}
   
In the next result,  we determine the {\it centralizer}  $\mathsf{C}_{\A_h}(x) = \{a \in \A_h \mid [a,x] = 0\}$ of
$x$ in $\A_h$ and then use that to describe the commutator $[\A_h, \A_h]$ when 
$\chara(\FF) = p > 0$.   

\begin{lemma}\label{L:commcent} Regard $\A_h\subseteq \A_1$  as in Conventions \ref{con:gens}.  
\begin{itemize}  \item[{\rm (i)}]  If  $\chara(\FF) = 0$,  then $\mathsf{C}_{\A_h}(x) = \DD = \FF[x].$ 
\item[{\rm (ii)}]  If $\chara(\FF) = p > 0$, then the following hold:
 \begin{enumerate}\item [\rm{(a)}]  $\mathsf{C}_{\A_h}(x) =  \displaystyle \FF[x, h^p y^p]=\bigoplus_{i \equiv 0 \modd p} \DD h^i y^i$.
\item[\rm{(b)}] $\displaystyle [x, \A_{h}] = \bigoplus_{i \not\equiv -1 \modd p} h\DD h^i y^i =  \bigoplus_{i=0}^{p-2}h \mathsf{C}_{\A_h}(x) h^i y^i$.
\item[\rm{(c)}] $[\hat y, \A_{h}] =  {\displaystyle \bigoplus_{i \ge 0}} \, \im \left( \frac{d}{dx} \right) h \hat y^i =  \displaystyle \bigoplus_{j \not\equiv -1\modd p}hx^{j}\FF[\hat y]$.
\end{enumerate}
\end{itemize}
\end{lemma}  

\begin{proof}   We first determine the centralizer $\mathsf C_{\A_1}(x)$.  
Suppose $a = \sum_{i=0}^n r_i y^i \in \mathsf{C}_{\A_1}(x)$, where $r_i \in \DD$ for all $i$.   Then
$0 = [a,x] = \sum_{i=1}^n i r_i y^{i-1}$.   When $\chara(\FF) = 0$, this forces $r_i = 0$ for all $i \geq 1$,
so that $a = r_0 \in \DD$.  Since $\DD \subseteq  \mathsf{C}_{\A_1}(x)$ is clear, we have $  \mathsf{C}_{\A_1}(x) =\DD$.  But then  $\mathsf{C}_{\A_h}(x) =  \mathsf{C}_{\A_1}(x) \cap \A_h = \DD$ to
give (i). 
When $\chara(\FF) = p > 0$, we deduce from this calculation that $r_i = 0$ for all $i \not \equiv 0 \modd p$.
Then $a = \sum_{i \equiv 0 \modd p}  r_i y^i \in   \FF[x, y^p]$, so $ \mathsf{C}_{\A_1}(x) \subseteq \FF[x, y^p]$. 
The reverse containment $\FF[x,y^p] \subseteq \mathsf{C}_{\A_1}(x)$ holds trivially, so $\mathsf{C}_{\A_1}(x) = \FF[x, y^p]$ (compare \cite[Proof of Prop.~1]{KA11}).
Now since $\mathsf C_{\A_1}(x) = \bigoplus_{i \equiv 0 \modd p} \DD y^i$, it follows that 
$$\mathsf C_{\A_h}(x) = \mathsf C_{\A_1}(x) \cap \A_h = \bigg \{ \sum_{i \equiv 0 \modd p} r_i y^i \, \bigg | \, r_i \in \DD h^i \bigg \}.$$  This establishes (a) of part (ii).  

(b) \, To describe $[x, \A_h] = [ \A_h, x]$ when $\chara(\FF) = p > 0$, note that for $a = \sum_{i \ge 0} r_i h^i y^i \in \A_h$, we can compute in $\A_1$ that 
$$[a, x] = \sum_{i \ge 0}  [r_i h^i y^i,  x] = \sum_{i \ge 0}  r_i h^i  [y^i,  x] = \sum_{i \not\equiv 0 \modd p}  i r_i h^i y^{i-1} = \sum_{i \not\equiv 0 \modd p} i h r_i h^{i-1} y^{i-1}.$$
Since $i \neq 0$ in $\FF$ as long as $i \not\equiv 0 \modd p$, we see that $\im ( \ad_x )$ is $\sum_{i \not\equiv -1 \modd p} h\DD h^i y^i$, and this sum is evidently direct.  The fact that 
$$\bigoplus_{i \not\equiv -1 \modd p} h \DD h^i y^i =  \bigoplus_{i=0}^{p-2}h\mathsf{C}_{\A_h}(x) h^i y^i$$
follows since $\mathsf C_{\A_h}(x) = \FF [x, h^py^p]$.

(c) \, For $a = \sum_{i \ge 0} r_i \hat y^i \in \A_h$, we have 
$$[\hat y, a] = \sum_{i \ge 0} [\hat y, r_i] \hat y^i = \sum_{i \ge 0} h r_i' \hat y^i,$$
and thus $\im (\ad_{\hat y} ) = \bigoplus_{i \ge 0} \im \left( \frac{d}{dx} \right) h \,  \hat y^i$.  Since $\im \left( \frac{d}{dx} \right) = \bigoplus_{j \not\equiv -1 \modd p} \FF x^j$, it follows that $\im (\ad_{\hat y} ) = \bigoplus_{j \not\equiv -1\modd p}hx^{j}\FF[\hat y]$.
\end{proof}

\begin{section}{The Normal Elements and Prime Ideals of $\A_h$}  \end{section}
 
Recall that an element $v\in \A_h$ is \emph{normal} if $v \A_h = \A_h v$.   In the polynomial algebra $\A_0 = \FF[x,y]$ every element of $\A_0$ is normal.   Similarly, the normal elements of the Weyl algebra $\A_1$ are precisely the central elements (compare Theorem~\ref{T:normalels}).  In general, for $h\notin\FF$, there are non-central normal elements in $\A_{h}$.  In this section, we determine the normal elements of $\A_h$ for arbitrary $h\neq 0$.   
Our starting point is 
\medskip  

\begin{lemma}\label{lem:h-factorsNormal}
Let $g$ be a factor of $h$ in $\DD=\FF[x]$.  Then $g$ is a normal element of $\A_h$.
\end{lemma}
\begin{proof}
Write $h = gf$ for $f \in \DD$.  Then 
$$\hat yg = g\hat y + hg' = g\hat y + gfg' = g(\hat y + fg') \in g \A_h$$ 
and $g \hat y =  (\hat y - fg')g \in \A_h g$.    
As $\A_h = \bigoplus_{i \ge 0}\DD \hat y^i$, it follows that $\A_h g \subseteq g \A_h$ and $g \A_h \subseteq \A_h g$, and so $g \A_h = \A_h g$.
\end{proof}
Since the product of two normal elements is normal, it is clear at this stage that products of powers of the prime factors of $h$ are normal elements of $\A_h$.   \medskip

Suppose
\begin{equation}\label{eq:hfactor} h =  \lambda \pr_{1}^{\alpha_{1}}\cdots \pr_{t}^{\alpha_{t}},\end{equation}  where $\lambda\in\FF^{*}$,  $\alpha_{i}\geq 1$ for all $i$, and the  $\pr_{i}\in \FF[x]$ are distinct monic prime
polynomials.  We can assume that the factors have been ordered so that the first ones $\pr_i$, for  $i \leq \ell \leq t$,  are the non-central prime divisors of $h$.   Our aim is to establish the following which generalizes (and includes) the result for the Weyl algebra.    

\begin{thm}\label{T:normalels}
Let $\pr_1, \ldots, \pr_\ell$ be the distinct  monic prime factors of $h$ in $\DD =\FF[x]$ that are not central in $\A_h$.  Then the normal elements of $\A_h$ are the elements of the form $\pr_{1}^{\beta_{1}} \cdots \pr_{\ell}^{\beta_{\ell}}z$,  where  $z\in \centh$.  If $\chara(\FF)=p> 0$, then the $\beta_i$ may be taken so that  $0\leq \beta_{i} < p$ for all $i$.
\end{thm}

The proof will use the next lemma. 
\begin{lemma}\label{L:normelts}  Let $\pr_1, \ldots, \pr_\ell$ be the distinct  monic prime factors of $h$ in $\DD$ that are not central in $\A_h$.  If $f$ divides $\delta(f) = hf'$ for $f \in \DD$, then there exist $w \in \DD\cap\centh$ and $\beta_i \in \mathbb Z_{\geq 0}$ for $i=1,\dots,\ell$ so that  $f = \pr_1^{\beta_1} \cdots \pr_\ell^{\beta_\ell}w$.     If $\chara(\FF) = p > 0$, the $\beta_i$ may be chosen so that $0 \leq \beta_i < p$ for all $i$.  \end{lemma}  

\begin{proof} The result is clear if $f \in \FF$, so assume $\degg f \geq 1$ and write $f = \mu \vv_1^{\gamma_1} \cdots \vv_n^{\gamma_n}$  where $\mu \in \FF^*$,  $\gamma_i \geq 1$ for all $i$,  and $\vv_1,\dots,\vv_n$ are distinct monic prime polynomials in $\FF[x]$.      Then  $f$ divides  
$$h f'   =  \mu h \sum_{i=1}^n \gamma_i \vv_1^{\gamma_1} \cdots \vv_i^{\gamma_i-1} \cdots \vv_n^{\gamma_n}\vv_i'.$$ 
This implies that  $\vv_j$ divides $\gamma_j \vv_j' h$ for all $j$. Then either $\vv_j$ divides $\gamma_j \vv_j'$ or $\vv_j$ divides $h$. If $\vv_j$ divides $\gamma_j \vv_j'$ then $\gamma_j \vv_j'=0$ which forces $\vv_j^{\gamma_j} \in \DD\cap\centh$, as $\left(\vv_j^{\gamma_j}\right)' =\gamma_j \vv_j'\vv_j^{\gamma_j-1}=0$. Otherwise,  $\vv_j = \pr_k$ for some non-central prime factor of $h$.   The last assertion in the lemma follows from the observation that when  $\chara(\FF) = p >0$,  then $r^p \in \FF[x^p]$ for all $r \in \DD$.   \end{proof} 
 
\begin{proof}[Proof of Theorem \ref{T:normalels}]
Assume $v\neq 0$ is normal in $\A_h$,  and write  $v = \sum_{i=0}^n f_i h^i y^i$, where $f_i \in \DD$ and $f_n \neq 0$. Then there exists  $a \in \A_h$ so that $v x = a v$, and from considering the coefficient of $y^n$, we see that $a \in \DD$,  and in fact $a = x$.  Thus $v x = x v$, and $v \in \mathsf C_{\A_h}(x)$.  Since $ hy \in \A_h$  by Lemma~\ref{lem:A_h-insideA_1}, there exists $b \in \A_h$ so that $v  (hy) = b v$ and, as above, we conclude that $b=hy-r$, for some $r \in \FF[x]$. The latter implies $[hy, v] = r v$.

Recall that  $\mathsf C_{\A_h}(x) = \DD=\FF[x]$ if $\chara (\FF) = 0$.  Hence,  in this case  $v \in \DD$, and $rv = [hy, v] = hv'$,   which implies by  Lemma \ref{L:normelts}
that $v = \zeta \pr_{1}^{\beta_{1}} \cdots \pr_{t}^{\beta_{t}}$,  where $\zeta \in \centh = \FF 1$ and $\beta_i \in \mathbb Z_{\geq 0}$ for all $i$.      
 
Thus,  for the remainder of the proof,  we assume that  $\chara(\FF) = p > 0$,  and because $v \in \mathsf{C}_{\A_h}(x)$,  we can write $v = \sum_{i \equiv 0 \modd p} f_i h^i y^i$.  We now know that  
$$
0=[hy, v] - r v = \sum_{i \equiv 0 \modd p} \left( [hy, f_i]-r f_{i}\right) h^i y^i = \sum_{i \equiv 0 \modd p} \left( hf_i'-r f_{i}\right) h^i y^i,
$$
which forces $r f_i = h f_i'$ for all $i \equiv 0 \modd p$.  This implies that $f_i$ divides $h f_i'$ for all such $i$, so by Lemma \ref{L:normelts}, there exist $w_{i}\in\FF[x^{p}]$ and integers $\beta_{1i}, \ldots, \beta_{\ell i}\in\{ 0,1, \ldots, p-1\}$ such that 
$$f_i=\pr_{1}^{\beta_{1i}} \cdots \pr_{\ell}^{\beta_{\ell i}}w_{i}.$$  

Fix $i,j$ and note $hf_i' f_j =r f_i f_{j} = hf_{j}'f_{i}$ holds, so that  $f_i' f_j = f_j' f_i$ since $h \neq 0$.  Now
$$
0=f_i' f_j - f_j' f_i = w_i w_j\sum_{k = 1}^\ell ( \beta_{k i} - \beta_{k j} ) \pr_1^{\varepsilon_1} \cdots \pr_{k -1}^{\varepsilon_{k - 1}} \pr_{k}^{\varepsilon_{k} -1} \pr_{k +1}^{\varepsilon_{k + 1}} \cdots \pr_\ell^{\varepsilon_\ell} \pr_{k}',
$$
where $\varepsilon_k= \beta_{k i} + \beta_{k j}$ for $k \in \{ 1, \ldots, \ell \}$. If $f_{i}, f_{j}\neq 0$, then $w_{i} w_{j}\neq 0$, and as a result we have 
$$\sum_{k = 1}^\ell ( \beta_{k i} - \beta_{k j} ) \pr_1^{\varepsilon_1} \cdots \pr_{k -1}^{\varepsilon_{k - 1}} \pr_{k}^{\varepsilon_{k} -1} \pr_{k+1}^{\varepsilon_{k+ 1}} \cdots \pr_\ell^{\varepsilon_\ell} \pr_{k}'  = 0,$$
which implies that $( \beta_{k i} - \beta_{k j} )\pr_k'$ is divisible by $\pr_k$ for each $k$.  Since $\pr_k$ is not central, $\pr_k' \neq 0$,  and thus $\beta_{k i} = \beta_{k j}$ for all $k$ and all $i,j$.  Letting $\beta_k$ be
that common exponent, we have  $f_i=\pr_{1}^{\beta_1} \cdots \pr_{\ell}^{\beta_\ell}w_i$ for each $i$, 
which says
$$v = \sum_{i \equiv 0 \modd p} f_i h^i y^i \ = \  \pr_{1}^{\beta_1} \cdots \pr_{\ell}^{\beta_\ell} \sum_{i \equiv 0 \modd p}  w_i h^i y^i  \ \in \  \pr_{1}^{\beta_1} \cdots \pr_{\ell}^{\beta_\ell} \centh.$$
\end{proof}

Several authors have studied the problem of determining simplicity criteria for Ore extensions $\DD[y, {\rm id}_\DD, \delta]$, and it is possible to address the simplicity of the algebras $\A_h$ by using the results of \cite{Jor75} or \cite[Thms.~3.2 and 3.2a]{CF75} for example.   Instead, we apply our results on normal and central elements of $\A_h$ to determine when an algebra $\A_h$ is
simple.

\begin{cor}\label{C:simplicity}
The algebra $\A_h$ is simple if and only if $\chara(\FF) = 0$ and $h \in \FF^*$.
\end{cor}
\begin{proof}
Suppose $\A_{h}$ is simple.  If $b \neq 0$ is a normal element of $\A_h$, then $b \A_h = \A_h b= \A_h$ by simplicity, so $b$ is a unit.  Since the units of $\A_h$ are the elements of $\FF^*$, we see that $h \in \FF^*$ by Lemma \ref{lem:h-factorsNormal}, and also $\mathsf Z(\A_h) = \FF 1$.  But then $\chara(\FF) = 0$ by Lemma~\ref{L:center}.  Conversely, if 
$\chara(\FF) = 0$ and $h \in \FF^*$,  then $\A_h$ is isomorphic to the Weyl algebra, and it  is well known that $\A_1$ is simple. 
\end{proof}

A (noncommutative) Noetherian
domain is said to be a \emph{unique factorization ring} (Noetherian UFR for
short), if every nonzero prime ideal contains a nonzero prime ideal generated by a normal element.
The {\it height} of a prime ideal is the largest  length of a chain of prime ideals contained in it (or is $\infty$ if no bound exists). 
A Noetherian UFR is said to be a \emph{unique factorization domain} (Noetherian UFD for short) 
if every height one prime factor is a domain.   These notions were introduced by Chatters and Jordan in~\cite{Ch84, ChJ86}.   If a Noetherian
domain satisfies the descending chain condition on prime ideals (e.g. if it has finite Gelfand-Kirillov dimension \cite[Cor.\ 8.3.6]{McR01}), then it is a Noetherian UFR if and only if every height one prime ideal is generated by a normal element. Recently, Goodearl and Yakimov \cite{GY12} have used the properties of noncommutative Noetherian UFDs to construct initial clusters for defining quantum cluster algebra structures on a noncommutative domain.

Since $\DD=\FF[x]$ is a principal ideal domain, \cite[Thm.\ 5.5]{ChJ86} trivially implies the first part of the following observation. The second part follows by~\cite[Thm.\ 9.24]{GW89}.

\begin{lemma}
$\A_{h}$ is a Noetherian UFR. If $\chara(\FF) =  0$, then $\A_{h}$ is a Noetherian UFD.
\end{lemma}

The algebra $\A_0 = \FF[x,y]$ is a Noetherian UFD for any field $\FF$.
We will see shortly that $\A_{h}$ is not a Noetherian UFD when $\chara(\FF) = p > 0$ and  $h \neq 0$. 

The next result describes the height one prime ideals of $\A_{h}$. It is known that over a field of prime characteristic the Weyl algebra $\A_{1}$ is Azumaya over its center (see~\cite[Th\'e.\ 2]{R73}), so in this case the prime ideals of $\A_{1}$ are in bijection with the prime ideals of $\cent1$. If $\degg h \geq 1$,  there may be prime ideals of $\A_{h}$ which are not centrally generated.

\begin{thm}\label{T:height1primes}
Let $\pr_1, \ldots, \pr_t$ be the distinct monic prime factors of $h$ in $\DD$, as in \eqref{eq:hfactor}.
For every $1\leq i\leq t$,  the normal element $\pr_{i}$ generates a height one prime ideal of $\A_{h}$,  and the corresponding quotient algebra  is a domain.
\begin{enumerate}
\item[{\rm (i)}] If $\chara(\FF) =  0$, these are all the height one prime ideals.
\item[{\rm (ii)}] If $\chara(\FF) = p > 0$,  then any nonzero irreducible polynomial in 
$\centh$ that (up to associates) is not of the form $\pr_{i}^{p}$ for any $1\leq i \leq t$ generates
a height one prime ideal.   These, along with the ideals generated by some $\pr_i$,  constitute all the height one prime ideals.
\end{enumerate}
\end{thm}

\begin{proof} First notice that each $\pr_{i}$ generates a prime ideal of $\A_{h}$, as the quotient algebra $\A_{h}/\pr_{i}\A_{h}$ is isomorphic to the commutative polynomial algebra $\left(\DD/\pr_{i}\DD\right)[\hat y]$ over the field $\DD/\pr_{i}\DD$. In particular, $\A_{h}/\pr_{i}\A_{h}$ is a domain,  and the prime ideal $\pr_{i}\A_{h}$ has height one  by the Principal Ideal Theorem (see \cite[Thm.\ 4.1.11]{McR01}).

Let $\PR$ be a height one prime ideal. Since $\A_{h}$ is a Noetherian UFR, it follows that $\PR=v\A_{h}$ for some normal element $v\neq 0$. Moreover, the primality of $\PR$ implies that $v$ is not a (non-trivial) product of normal elements. Thus, Theorem~\ref{T:normalels} implies that either $v$ is an irreducible factor of $h$ or a central element which is irreducible as an element in $\centh$.   When $\chara(\FF) =  0$,  then $v$ must be an irreducible factor of $h$, as $\centh = \FF 1$, which proves (i).

For the remainder of the proof assume $\chara(\FF) = p > 0$. Note that if $z\in\centh$ is of the form $\xi \pr_{i}^{p}$ for some $i$ and some $\xi \in\FF^{*}$, then $z\A_{h}$ is not a prime ideal. So it remains to show that if $z$ is an irreducible polynomial in $\centh$,  which is not of the form $\xi \pr_{i}^{p}$ for $1\leq i\leq t$ and $\xi \in\FF^{*}$, then $z\A_{h}$ is a height one prime ideal. We can further assume $z$ is not an irreducible factor of $h$, as this case has already been considered. Let $\PR\supseteq z\A_{h}$ be a minimal prime over $z\A_{h}$. By the Principal Ideal Theorem, $\PR$ has height one,  and thus $\PR=v\A_{h}$ for some normal element $v$.

Suppose first that $v$ is an irreducible factor of $h$, say $v=\pr_{n}$. Then $z\in \PR=v\A_{h}$, so $z=\pr_{n}a$ for some $a\in\A_{h}$. Write $a=\sum_{i\geq 0}r_{i}h^{i}y^{i}$ with $r_{i}\in\FF[x]$, so that $z=\pr_{n}a=\sum_{i\geq 0}\pr_{n}r_{i}h^{i}y^{i}$. As $z$ is central, we must have $r_{i}=0$ if $i\not\equiv 0\modd p$ and $\pr_{n}r_{i}\in\FF[x^{p}]$ for all $i\equiv 0\modd p$. Fix $j$ with $j\equiv 0\modd p$ and $r_{j}\neq 0$. Let $\vv_1^{\gamma_1} \cdots \vv_m^{\gamma_m}$ be the prime decomposition of $\pr_{n}r_{j}$ in $\FF[x]$, with $\vv_{1}=\pr_{n}$. Then $\gamma_{1}\geq 1$ and since $\pr_{n}r_{j}\in\FF[x^{p}]$, it follows that $\vv_i^{\gamma_i}\in\FF[x^{p}]$ for all $1\leq i\leq m$. In particular, $\pr_{n}^{\gamma_{1}}\in\FF[x^{p}]$, so that either $\gamma_{1}\equiv 0\modd p$ or $\pr_{n}\in\FF[x^{p}]$. If the latter holds,  then $z=\pr_{n}a$ implies that $a\in\centh$.  The irreducibility of $z$ in $\centh$ implies that $a\in\FF^{*}$,  and thus $z$ is an irreducible factor of $h$, which contradicts our previous assumption. So it must be that $\gamma_{1}\equiv 0\modd p$. As $\gamma_{1}\geq 1$, it follows that $\gamma_{1}\geq p$ and $\pr_{n}^{p}$ divides $\pr_{n}r_{j}$. Since $j\equiv 0\modd p$ was arbitrary subject to the restriction that $r_{j}\neq 0$, we deduce that  $z=\pr_{1}^{p}c$ for some $c \in\centh$. The irreducibility of $z$ in $\centh$ again implies that $z$ is a scalar multiple of $\pr_{n}^{p}$, which violates our assumptions on $z$. 

It follows from the arguments in the preceding paragraph that $v$ is not an irreducible factor of $h$. Hence $v\in\centh$,  and again we deduce that $z=va$ for some $a\in\centh$. Thus, as $z$ is irreducible in $\centh$, it must be that $a\in\FF^{*}$ and $z\A_{h}=v\A_{h}=\PR$ is a height one prime ideal.
\end{proof}    

\begin{cor}
Assume $\chara(\FF) = p > 0$. Then $\A_{h}$ is not a Noetherian UFD.
\end{cor}

\begin{proof}
By Theorems~\ref{L:center} and~\ref{T:height1primes}, the element $h^{p}y^{p}$ generates a height one prime ideal of $\A_{h}$, as it is irreducible in $\centh$ and it is not a power of a factor of $h$. However, by \eqref{eq:centgen} we have $h^p y^p =\left( \hat y^{p-1} - \frac{\delta^p(x)}{h}\right) \hat y$.  Yet neither one of these two factors is in $h^{p}y^{p}\A_{h}$, by considering the degree in $y$ of an element in $h^{p}y^{p}\A_{h}$. Thus, the prime ring $\A_{h}/h^{p}y^{p}\A_{h}$ is not a domain.
\end{proof}

\begin{remark}
Since $\A_{h}$ has Gelfand-Kirillov dimension $2$, it follows from \cite[Cor.\ 8.3.6]{McR01} that the possible values for the height of a prime ideal $\PR$ of $\A_h$  are $0,1$,  and $2$. The zero ideal is prime and is thus the unique prime ideal of height zero. The height one prime ideals are given in Theorem \ref{T:height1primes}.   The height two prime ideals of $\A_h$ must be maximal,  and no height one prime ideal of $\A_{h}$ can be maximal. Indeed,  for the height one prime ideals of the form $\pr_{i}\A_{h}$,  $1\leq i\leq t$, the quotient $\A_{h}/\pr_{i}\A_{h}$ is a commutative polynomial algebra. When $\chara(\FF)=p>0$,  the center  $\centh$ is a polynomial algebra in two variables, so if $v$ is an irreducible polynomial in $\centh$ as in Theorem \ref{T:height1primes} (ii) above, it follows that any maximal ideal of $\centh$ containing $v$ induces a maximal ideal of $\A_{h}$ strictly containing $v\A_{h}$.

Hence, the height two prime ideals of $\A_{h}$ are precisely the maximal ideals of $\A_{h}$, and can be identified with the maximal ideals of $\A_{h}/\PR$, as $\PR$ ranges through the height one prime ideals. In particular, if $\chara(\FF)=0$ and the prime factors of $h$ in $\FF[x]$ are linear, then the height two prime ideals of $\A_{h}$ are the ideals generated by $x-\lambda$ and $q(\hat y)$, where $\lambda\in\FF$ is a root of $h$ and $q(\hat y)\in\FF[\hat y]$ is an irreducible polynomial.
\end{remark}

\begin{section}{Automorphisms of $\A_h$} \end{section}
Extending results of Dixmier \cite{Dix68} on the automorphisms of the Weyl algebra $\A_1$,  
Bavula and Jordan \cite{BJ01}  considered isomorphisms and automorphisms of generalized  Weyl algebras over polynomial algebras  of
characteristic 0.    Alev and Dumas \cite{AD97} initiated  the study of automorphisms of Ore extensions over the polynomial algebra $\DD = \FF[x]$, and the results in \cite{AD97} have been further developed in the recent work \cite{Gaddis12} of Gaddis.   In Theorem \ref{T:isos:A_h}, we summarize results from \cite{AD97}  that pertain to the algebras $\A_h$ studied here, but suitably interpreted in the notation of the present paper.  
Since one of those results assumes that $\chara(\FF) = 0$, we first prove Lemma \ref{lem:isomRtoR}, which can be used to remove that characteristic assumption.   This will enable us to prove our main results, Theorems \ref{T:autosA_hprod} and \ref{T:Pinfo}, which give a complete  description of the automorphisms of $\A_h$ over arbitrary fields. 
\medskip

\begin{lemma}\label{lem:isomRtoR}
If $\theta : \A_h \to \A_g$ is an isomorphism, then $\theta (h) = \lambda g$ for some $\lambda \in \FF^*$.
\end{lemma}
\begin{proof}
Let $\BBh$ be the ideal of $\A_h$ minimal with the property that $\A_h/\BBh$ is commutative. Then $[y,x] = 0$ in the quotient $\A_h / \BBh$, so it follows that $h \in \BBh$.  The element $h$ is normal in $\A_h$ and $h \A_h \subseteq \BBh$, so the minimality of $\BBh$, with the fact that $\A_h / h \A_h$ is commutative,  implies that $h \A_h = \BBh$.  Similar reasoning shows that $\BBg = g \A_g$ is the ideal of $\A_g$ minimal with the property that $\A_g / \BBg$ is commutative.  As $\BBh$ and $\BBg$ are obviously characteristic ideals, it follows that $\theta (\BBh)=\BBg$.     
Since $\A_g$ is a domain and $g \A_g = \BBg = \theta (\BBh) = \theta (h) \A_g$, we have that $\theta (h)=\lambda g$ for some $\lambda\in\FF^*$.
\end{proof}

Now with  Lemma \ref{lem:isomRtoR}, the argument in the proof  \cite[Prop.~3.6]{AD97}  can be extended to
arbitrary fields, and  as a result,  we have the following. 
\begin{thm}\label{T:isos:A_h}
Let $g, h\in\FF[x]$.
\begin{enumerate}
\item[{\rm (i)}]  $\A_h$ is isomorphic to $\A_g$ if and only if there exist $\alpha, \beta, \nu \in\FF$, with $\alpha \nu \neq 0$ such that $\nu g(x)= h(\alpha x+\beta)$.   In particular, if $\A_h$ is isomorphic to $\A_g$,  then $g$ and $h$ have the same degree. 
\item[{\rm (ii)}]  Suppose  $\degg h \geq 1$. Let $\omega$ be an automorphism of $\A_h$. Then there exist $\alpha, \beta\in\FF$, with $\alpha\neq 0$,  and  $f(x)\in\FF[x]$ such that 
$$\omega (x)=\alpha x+\beta, \quad \omega (\hat y)=\alpha^{\degg h-1}\hat y+f(x), \quad \text{and} \quad h(\alpha x+\beta)=\alpha^{\degg h} h(x).$$
\end{enumerate}
\end{thm}
 
\begin{subsection}{Automorphisms of $\A_h$  \newline Definitions and the Decomposition}\end{subsection}

If $h\in\FF$,  the automorphism group of $\A_h$ is known \cite{wVDK53, Dix68, MakLim84} (see also the discussion in Sec. \ref{sec:automWeyl} below), so in what follows,  we assume $\deg h\geq 1$.   In view of Theorem \ref{T:isos:A_h}, we introduce the following definitions.     Let
\begin{equation}\label{eq:defP}  \mathbb P = \{(\alpha,\beta) \in \FF^* \times \FF \mid  h(\alpha x +\beta) = \alpha^{\degg h} h(x)\}. \end{equation} 
It is easy to verify that each pair $(\alpha,\beta) \in \mathbb P$ determines an  automorphism $\tau_{\alpha,\beta}$ of $\A_h$ whose values on $x$ and $\hat y$ are given by
\begin{equation} \label{eq:sigab} \tau_{\alpha,\beta}(x) = \alpha x + \beta, \qquad \tau_{\alpha,\beta}(\hat y) = \alpha^{\mathsf{deg}h-1}\hat y. \end{equation}  The pair $(\alpha^{-1},-\beta \alpha^{-1})$ belongs to $\mathbb P$ whenever $(\alpha,\beta)$ does, and $\tau_{\alpha,\beta}^{-1} = 
\tau_{\alpha^{-1}, -\beta\alpha^{-1}}$.

Each  $f\in\FF[x]\subseteq \A_h$ determines an automorphism $\phi_f$  of $\A_h$ defined by 
\begin{equation} \phi_f (x)=x, \qquad \phi_f (\hat y)= \hat y+f  \end{equation} 
and having inverse $\phi_{-f}$. Furthermore, $\{ \phi_f \mid f\in\FF[x]  \}$ is a subgroup of $\aut (\A_h)$, isomorphic to the additive group $\FF[x]$. One important example is the automorphism $\phi_{h'}$ with   $\phi_{h'} (x) = x$ and $\phi_{h'} (\hat y) = \hat y + h'$. 
The normality of  the element $h \in \A_h$ (see Lemma \ref{lem:h-factorsNormal}) implies that this automorphism has the property that 
\begin{equation}
ah = h \phi_{h'} (a)
\end{equation}
for all $a \in \A_h$  (compare (\ref{eqn:(yHat)(h)})). 
\medskip

\begin{thm}\label{T:autosA_hprod}
Suppose $\degg h \ge 1$, and let the set $\PP$ and the automorphisms $\tau_{\alpha,\beta}$ for $(\alpha,\beta) \in \PP$ be as in \eqref{eq:defP} and \eqref{eq:sigab}. 
\begin{itemize}
\item[{\rm (i)}] If $\omega$ is an automorphism of $\A_h$,
then there exist 
$(\alpha, \beta) \in \mathbb P$ and $f \in \FF[x]$ such that 
 $\omega = \phi_f \circ \tau_{\alpha, \beta}$.
 \item[{\rm (ii)}]  $\tau_{\alpha,\beta} = \phi_f$ for some $(\alpha,\beta) \in \mathbb P$ and $f \in \FF[x]$
 if and only if $\alpha = 1, \beta = 0$ and $f = 0$.       
  \item[{\rm (iii)}]  If $(\alpha,\beta) \in \mathbb P$,  $\alpha \neq 1$,  and  $\alpha^\ell = 1$  for some $\ell \ge 2$,  then $\tau_{\alpha,\beta}^\ell = \mathsf{id}_{\A_h}$.
 \item[{\rm (iv)}]
 The abelian subgroup $\{ \phi_f \mid f \in \FF[x]\}$, which we identify with \hbox{\rm ($\FF[x],+$)},   is a normal subgroup of $\aut(\A_h)$.
 \item[{\rm (v)}]  $\aut(\A_h) = \FF[x] \rtimes \tau_{\mathbb P}$, where $\tau_\mathbb P :=
 \{ \tau_{\alpha,\beta} \mid (\alpha,\beta) \in \mathbb P\}$ and $\tau_{\mathbb P}$ is a subgroup
 of $\aut(\A_h)$.  
 \end{itemize}
\end{thm}

\begin{proof}  Part (i) is immediate from Theorem \ref{T:isos:A_h}. 
If $\tau_{\alpha,\beta} = \phi_f$ for some $(\alpha,\beta) \in \mathbb P$ and $f \in \FF[x]$, then
$\alpha x + \beta = \tau_{\alpha,\beta}(x) = \phi_f(x) = x$, which implies  $\alpha = 1$
and $\beta = 0$.  Then, $\hat y = \alpha^{\degg h-1}\hat y = \tau_{\alpha,\beta}(\hat y) = \phi_f(\hat y) = \hat y+f(x)$, to
force $f = 0$.   The converse is clear, since 
$\tau_{1,0} = \mathsf{id}_{\A_h} = \phi_0$.  
 
Suppose $(\alpha,\beta), (\gamma,\varepsilon) \in \mathbb P$.  Then 
$(\alpha \gamma, \beta\gamma + \varepsilon) \in \mathbb P$, as
$$h(\alpha\gamma x +\beta\gamma + \varepsilon) = h(\gamma (\alpha x + \beta) +\varepsilon) =  \gamma^{\mathsf{deg}h} h(\alpha x + \beta) 
= (\alpha \gamma)^{\mathsf{deg}h} h(x).$$    Moreover, 
\begin{equation}\label{eq:sigmamult} \tau_{\alpha,\beta}\circ \tau_{\gamma,\varepsilon} =  \tau_{\alpha \gamma,  \beta\gamma+\varepsilon}.\end{equation}
Consequently,  $\tau_\mathbb P = \{ \tau_{\alpha,\beta} \mid (\alpha,\beta) \in \mathbb P\}$ is a subgroup of $\aut(\A_h)$.  
Now \eqref{eq:sigmamult} implies  $\tau_{\alpha, \beta}^\ell = \tau_{\alpha^\ell, ( 1+\alpha + \cdots +\alpha^{\ell- 1})\beta}$ for all $\ell \geq 1$.
Hence, if $\alpha^\ell = 1$ and $\alpha \neq 1$, then  $\tau_{\alpha, \beta}^\ell =  \tau_{1, 0} = \mathsf{id}_{\A_h}$.

Direct calculation shows that 
\begin{equation}\label{eq:autnorm}   
\tau_{\alpha,\beta}^{-1}\circ \phi_f \circ \tau_{\alpha, \beta}(x) = x,  \ \ \  
\tau_{\alpha,\beta}^{-1}\circ \phi_f \circ \tau_{\alpha, \beta}(\hat y) = \hat y+ \alpha^{\mathsf{deg}h-1} f\big(\alpha^{-1}(x-\beta)
\big).\end{equation}
Thus,  $\tau_{\alpha,\beta}^{-1}\circ \phi_f \circ \tau_{\alpha, \beta} = \phi_g$,  where
$g(x) = \alpha^{\mathsf{deg}h-1} f\big(\alpha^{-1}(x-\beta)
\big).$  Since every automorphism is a product of automorphisms  in the subgroups $\FF[x]$ and $\tau_{\mathbb P}$, we have that the subgroup $\FF[x]$ is normal in $\aut(\A_h)$. 
Part (v) follows then, since the two subgroups have trivial intersection by (ii). 
\end{proof}
\medskip

The automorphism group $\aut(\A_h)$ will be 
completely determined once we  establish conditions for a pair
$(\alpha, \beta)$ to belong to $\mathbb P$.  This will of course depend on the polynomial $h$.

\begin{subsection}{The Subgroup $\tau_\PP$} \end{subsection} 

In the following, we adopt the notation 
\begin{equation}\label{eq:Gdef} \G = \{\nu \in \FF \mid (1,\nu) \in \PP\} \quad \text{and} \quad \tau_{1,\G} = \{ \tau_{1, \nu} \mid \nu \in \G \}. \end{equation}

\begin{lemma}\label{L:PandG} Suppose $\degg h\geq 1$. Let the set $\PP$ and 
the automorphisms $\tau_{\alpha,\beta}$ for $(\alpha,\beta) \in \PP$ be as in \eqref{eq:defP} and
\eqref{eq:sigab}.   
\begin{itemize}
\item[{\rm (1)}]   $\G$ is a finite subgroup of $(\FF,+)$,
which is equal to $\{0\}$ when $\chara(\FF) = 0$.   
\item[{\rm (2)}]   If $(\alpha,\beta) \in \PP$ and $(\alpha,\tilde \beta) \in \PP$, then 
$\tau_{\alpha,\tilde \beta} = \tau_{\alpha,\beta}\circ \tau_{1,\nu}$ where $\nu = \tilde \beta - \beta  \in \G$.
In particular, $\tilde \beta = \beta$ must hold when $\G = \{0\}$.  
\item[{\rm (3)}]   If $(\alpha,\beta) \in \PP$ and $\nu \in \G$,   then
$$\tau_{\alpha,\beta}^{-1}\circ \tau_{1,\nu}\circ \tau_{\alpha,\beta} = \tau_{1,\alpha\nu},$$ so $\alpha \nu \in \G$.  
\item[{\rm (4)}] 
$\mathsf{N}: = \FF[x] \rtimes \tau_{1,\G}$ is a normal subgroup of $\mathsf{Aut}_\FF(\A_h)$, which equals
$\FF[x]$ when $\chara(\FF) = 0$.   
\end{itemize}
\end{lemma} 

\begin{proof}  (1) It follows from \eqref{eq:sigmamult}  that  $\tau_{1,\nu}\circ \tau_{1,\tilde \nu} = \tau_{1,\nu + \tilde \nu}$ whenever
$\nu,\tilde \nu \in \G$, so $\G$ is a subgroup of $(\FF,+)$.   Let $\overline{\FF}$ denote the algebraic closure of $\FF$, and let $\lambda \in \overline \FF$ be a root of $h(x)$.  Then $\{ \lambda + \nu \mid \nu \in \mathbb G \}$ consists of roots of $h(x)$, so it is evident that $\mathbb G$ is finite provided $h\notin\FF$.
When  $\chara (\FF) = 0$, then $\mathbb G = \{ 0 \}$,  as this is the only finite subgroup of $(\FF,+)$.   

(2)  Assume $(\alpha,\beta) \in \PP$ and $(\alpha, \tilde \beta) \in \PP$.    Because  $\tau_\PP$ is a group,
$$\tau_{\alpha,\beta}^{-1} \circ \tau_{\alpha, \tilde \beta} =  \tau_{\alpha^{-1}, - \alpha^{-1}\beta} \circ \tau_{\alpha, \tilde \beta} = \tau_{1,\tilde \beta-\beta} \in \tau_\PP.$$
Thus  $\nu: = \tilde \beta - \beta \in \G$.

(3) Suppose $(\alpha,\beta), (1,\nu) \in \PP$.    Then since $\tau_{\alpha, \beta}^{-1} = \tau_{\alpha^{-1}, - \alpha^{-1} \beta}$, \eqref{eq:sigmamult} gives that 
$$\tau_{\alpha,\beta}^{-1}\circ\tau_{1,\nu}\circ \tau_{\alpha,\beta} = \tau_{1,\alpha\nu},$$ 
as desired. 

(4)  From \eqref{eq:autnorm}  we know that
$$\tau_{\alpha,\beta}^{-1}\circ \phi_f \circ \tau_{\alpha,\beta} = \phi_g,$$ where $g = \alpha^{\degg h-1}f\big(\alpha^{-1}(x-\beta)\big),$
which implied the normality of the subgroup $\{\phi_f \mid f \in \FF[x]\}$ in $\aut(\A_h).$ (We identified this subgroup  with $\FF[x]$.)
 Part (3) shows that conjugation by the elements $\tau_{\alpha,\beta}$ for
$(\alpha,\beta)\in \PP$ leaves $\tau_{1,\G} = \{ \tau_{1,\nu} \mid \nu \in \G\}$ invariant.  Hence,
$\FF[x] \rtimes \tau_{1,\G}$ a normal subgroup of $\aut(\A_h)$.   Since $\tau_{1,\G}$
just consists of $\tau_{1,0} = \mathsf{id}_{\A_h}$ whenever $\G = \{0\}$,   this normal subgroup equals  $\FF[x]$ when $\G = \{0\}$ (for example, when $\chara(\FF) = 0$).     
\end{proof}   \smallskip

\begin{remark}\label{R:coprime}
From (3) of Lemma \ref{L:PandG}, it follows that $\tau_{1, \G}$ is a normal subgroup of $\tau_\PP$ and that $\tau_\PP / \tau_{1, \G}$ acts on $\G$ via $(\tau_{\alpha, \beta} \tau_{1, \G}). \nu = \alpha \nu$.  If $\G \setminus \{ 0 \}$ is nonempty, then this formula shows that $\tau_\PP / \tau_{1, \G}$ acts faithfully on $\G \setminus \{ 0 \}$, and therefore $| \G|-1$ is divisible by $| \tau_\PP / \tau_{1, \G} |$.
\end{remark}
\smallskip

The group $\FF[x] \rtimes \tau_{1,\G}$ may not be all of $\aut(\A_h)$,  and in that situation,   there exists some $(\alpha,\beta) \in \PP$ with $\alpha \neq 1$ so that $\tau_{\alpha,\beta} \in \aut(\A_h)$.   The next result draws conclusions in that case.   

\begin{thm}\label{T:Pinfo}   Assume $h$ has $k$ distinct roots in $\overline \FF$ for $k \geq 1$.
\smallskip \smallskip

\noindent  {\rm (Case $k = 1$)} \  Let $\lambda$ be the unique root of $h$ in $\overline \FF$. 
\begin{enumerate}
\item [{\rm (a)}]  If  $\lambda \in \FF$,  then $\PP = \{ (\alpha, (1 - \alpha) \lambda ) \mid \alpha \in \FF^* \}$,  $\tau_\PP \cong \FF^*$, and  $\aut(\A_h) = \FF[x] \rtimes \FF^*$, where for all $f \in \FF[x]$ and $\alpha \in \FF^*$,   
$$\tau_{\alpha, (1 - \alpha) \lambda}^{-1} \circ \phi_f \circ \tau_{\alpha, (1-\alpha)\lambda} =  \phi_g \ \ \ \hbox{ with }$$
 $$g(x) = \alpha^{\degg h-1}f(\alpha^{-1}x-(\alpha^{-1}-1)\lambda).$$  
 \item [{\rm (b)}]  If $\lambda \notin \FF$,  then $\aut(\A_h) = \FF[x]$.
\end{enumerate} 
\noindent  {\rm (Case $k \geq 2$)} \  The group  $\tau_\PP/\tau_{1,\G}$ is a finite cyclic group.
 In particular,  when $\tau_{\PP} \neq \tau_{1,\G}$, then  $\tau_{\PP} =  \tau_{1,\G} \rtimes \langle \tau_{\alpha,\beta} \rangle$,  for some $(\alpha,\beta) \in \PP$ with $\alpha \neq 1$ such that  either $\alpha^{k-1} = 1$ or $\alpha^k = 1$, 
and  $\tau_{\alpha,\beta}^{-1}\circ \tau_{1,\nu} \circ \tau_{\alpha,\beta} = \tau_{1,\alpha \nu}$ for all $\nu \in \G$.
Thus, $\aut(\A_h) \cong \mathsf{N} \rtimes \langle \tau_{\alpha,\beta} \rangle$ where
$\mathsf{N} = \FF[x] \rtimes \tau_{1,\G}$.    
\end{thm} 

\begin{proof}  Assume $(\alpha,\beta) \in \PP$.  By the definition of $\PP$, the affine bijection $\sigma_{\alpha,\beta}$ of $\overline \FF$ given by  $\sigma_{\alpha,\beta}(\lambda) = \alpha \lambda + \beta$ permutes the roots of $h(x)$ in such a way that the corresponding multiplicities are preserved.   Thus $\lambda + \nu$ is a root of $h(x)$ whenever $\lambda$ is a root of $h(x)$ and $\nu \in \mathbb G$, so it follows that $\mathbb G = \{ 0 \}$ when $k = 1$.

When $h(x)$ has the form $h(x)=\gamma(x-\lambda)^n$ with $\lambda\in\FF$,  then  $(\alpha,(1-\alpha) \lambda) \in \mathbb P$ for any $\alpha \in \FF^*$, as $h(\alpha x+ (1-\alpha)\lambda) = \gamma (\alpha x - \alpha \lambda)^n = \alpha^n \gamma (x-\lambda)^n = \alpha^n h(x)$.     Conversely, if   $(\alpha, \xi) \in \mathbb P$, for some $\xi$, then $\xi = (1 - \alpha)\lambda$ must hold  because $(\alpha, (1 - \alpha)\lambda) \in \mathbb P$ and $\mathbb G = \{ 0 \}$.   Since $\tau_{\alpha, (1-\alpha)\lambda}\circ \tau_{\mu,(1-\mu)\lambda} = \tau_{\alpha\mu, (1-\alpha\mu) \lambda}$, we may identify the group $\tau_\PP$ with $\FF^*$ in this case.   Thus, $\aut(\A_h) = \FF[x] \rtimes \FF^*$.   The product formula appearing in (a)  follows from \eqref{eq:autnorm}. Hence, the theorem holds when $k = 1$ and $\lambda\in\FF$. 

Suppose now that $k = 1$ and $\lambda\notin\FF$. Then $\sigma_{\alpha,\beta}(\lambda) = \lambda$ whenever $(\alpha, \beta) \in \PP$, so that $(1-\alpha)\lambda=\beta$. If $\alpha\neq 1$ then $\lambda = \beta/(1-\alpha)\in\FF$, which contradicts our hypothesis. Thus, $\alpha=1$ and $\beta=0$, which proves that $\tau_{\PP}$ is trivial and $\aut(\A_h) = \FF[x]$ in this case.

We now assume $k \ge 2$.  Suppose $\lambda\in \overline \FF$ is a root of $h(x)$.  Orbits under the $\sigma_{\alpha, \beta}$ are finite, so if $(\alpha, \beta) \in \PP$, there must be a minimal $j \geq 1$ so that $\sigma_{\alpha,\beta}^j(\lambda) = \lambda$.   It follows that $\lambda = \alpha^j \lambda +(1+\alpha+\cdots + \alpha^{j-1})\beta$; that is,  $(1-\alpha^j)\lambda = (1+\alpha+\cdots + \alpha^{j-1})\beta$.  If   $\alpha$ is not a $j$th root of 1, then we  obtain $\lambda = \beta/(1-\alpha)$.     Since the root $\lambda$ was chosen arbitrarily,   this shows that if $(\alpha,\beta) \in \PP$ for some $\alpha$ which is not a root of unity, then $h(x)$ has a unique root $\lambda = \frac{\beta}{1-\alpha} \in \FF$, and  $h(x) = \gamma(x-\lambda)^n$ for some $\gamma \in \FF^*$ and $n\geq 1$.

Assume that $\tau_{\PP} \neq \tau_{1,\G}$ and that $(\alpha, \beta) \in \PP$ with $\alpha$  a primitive $\ell$th root of unity for some $\ell \geq 2$.   We want to show that $\ell$ divides $k$ or $k-1$.   As before, let  $\lambda\in\overline{\FF}$ be a root of $h$, and suppose the orbit of $\lambda$ under the action of the cyclic group $\langle \sigma_{\alpha, \beta} \rangle$ generated by $\sigma_{\alpha, \beta}$ has cardinality $j$. We will argue that $j\in\{ 1, \ell\}$. The integer $j\geq 1$ is the smallest positive integer such that $\sigma^{j}_{\alpha, \beta}(\lambda)=\lambda$, which is equivalent to 
\begin{equation*}
(\alpha^j-1) \lambda+\beta(1+\alpha+\cdots+\alpha^{j-1})=0,
\end{equation*}
as we have seen above.   If $j<\ell$, then $\alpha^{j}\neq 1$, so we can divide by $\alpha^j-1$ and get  $\lambda=\frac{\beta}{1-\alpha}$ and $j=1$. Now notice that $\sigma^{\ell}_{\alpha, \beta}(\lambda)=\alpha^\ell \lambda+ \left(\frac{1-\alpha^\ell}{1-\alpha}\right)\beta=\lambda$, so $j\leq\ell$. Thus $j\in\{ 1, \ell\}$.

Hence, the orbits of this action of $\langle \sigma_{\alpha, \beta} \rangle$ on the roots of $h(x)$  have size either $1$ or $\ell$. Let $r$ be the number of orbits of size $1$ and $q$ the number of orbits of size $\ell$. It follows that $k=r+q\ell$, so $\ell$ divides $k-r$. If the orbits of two roots $\lambda$ and $ \tilde \lambda$ have size $1$,  then $\lambda=\frac{\beta}{1-\alpha}=\tilde \lambda$, so $r \leq1$.   Thus, either $r = 0$ and $\ell$ divides $k$ or $r = 1$ and $\ell$ divides $k-1$.

By  \eqref{eq:sigmamult}, the projection map $\psi: \tau_\PP \rightarrow \FF^*$ given by $\psi(\tau_{\mu,\nu}) = \mu$ is a group homomorphism with kernel $\tau_{1,\G}$. The image is  a finite subgroup of $\FF^*$, since $\FF^*$ has only finitely many $k$ and $k-1$ roots of unity.  As finite subgroups of $\FF^*$ are cyclic, we have that $\tau_\PP/\tau_{1,\G}$ is generated by a coset $\tau_{\alpha,\beta}\, \tau_{1,\G}$ for some $(\alpha,\beta) \in \PP$ such that $\alpha^{k-1} = 1$ or $\alpha^k = 1$ (but not both).  The rest of the statements follow from  Lemma \ref{L:PandG} and Theorem~\ref{T:autosA_hprod}.    
\end{proof}  
In the next result, we will use the notation $\sigma_{\PP} = \{\sigma_{\zeta,\varepsilon} \mid (\zeta,\varepsilon) \in \PP\}$ for the group of affine maps on $\overline \FF$ determined by $\PP$,  and $\sigma_{1,\G}$ for the subgroup determined by $\G$,
along with the fact that these groups act on the set of roots of $h$ in $\overline \FF$.   

\begin{cor}\label{C:Pinfo}  Assume $h$ has $k$ distinct roots in $\overline \FF$  for $k \geq 1$.   
\smallskip \smallskip

\noindent  {\rm (Case $k = 1$)} \  Let $\lambda$ be the unique root of $h$ in $\overline \FF$. 
\begin{enumerate}
\item [{\rm (a)}]  If  $\lambda \in \FF$,  then $\aut(\A_h) = \FF[x] \rtimes \FF^*$, where $\FF^*$ is identified with the 
group $\{\tau_{\alpha,(1-\alpha)\lambda} \mid \alpha \in \FF^*\}$.  
\item [{\rm (b)}]  If $\lambda \notin \FF$,  then $\aut(\A_h) = \FF[x]$.
\end{enumerate} 
\noindent  {\rm (Case $k \geq 2$)} \   Either
\begin{enumerate}
\item [{\rm (a)}]  $\aut(\A_h) \cong \FF[x] \rtimes \tau_{1,\G}$, and there exist  orbit representatives  
$\lambda_i,  i \in \tt I$,  for the action of  $\sigma_{1,\G}$ on 
the roots of $h$,  so that $h= \gamma \prod_{i \in {\tt I}}   h_i^{n_i}$,  where $\gamma \in \FF^*$, $n_i \geq 1$, and
$h_i(x) = \prod_{\nu \in \G} \big (x - \sigma_{1, \nu}(\lambda_i)\big)    = \prod_{\nu \in \G} \big(x - (\lambda_i + \nu)\big)$
for all $i \in \tt I$;  
\end{enumerate}   
 or  there exists $(\alpha,\beta) \in \PP$, where  $\alpha$ is a primitive $\ell$th root of unity  for some $\ell > 1$ such that  $\ell$ divides $k-1$ or $k$,  and  $\aut(\A_h) \cong ( \FF[x]  \rtimes \tau_{1,\G}) \rtimes \langle \tau_{\alpha,\beta} \rangle$.  
\begin{enumerate}
\item [{\rm (b)}]   If $\ell$ divides $k-1$,  then $\lambda_ 0: = \beta/(1-\alpha)$ is 
a root of $h(x)$ in $\FF$.  There are roots $\lambda_i, \ i \in \tt I$,  of $h$ in $\overline \FF$  so that $\{\lambda_i \mid i \in \tt I\}\cup \{\lambda_0\}$ are orbit representatives  for the action of $\sigma_{\PP}$ on the roots of $h$;  integers $n_i \geq 1$ for $i \in \tt I \cup \{0\}$;  and  
$\gamma \in \FF^*$  so that
$h = \gamma h_0^{n_0} \prod_{i \in {\tt I}} h_i^{n_i}$,  where 
 \begin{eqnarray}\label{eq:kpminus1}&& \hspace{-.75 truein}  h_0(x) =  \prod_{\nu \in \G} \big (x - \sigma_{1, \nu}(\lambda_0)\big)    = \prod_{\nu \in \G} \big(x - (\lambda_0+ \nu)\big) \\
&& \hspace{-.75 truein}  h_i(x) = \prod_{(\zeta,\varepsilon) \in \PP}  \big(  x - \sigma_{\zeta,\varepsilon}(\lambda_i)\big) 
= \Bigg(\prod_{\nu \in \G}  \prod_{j=0}^{\ell-1} \bigg(x-\left(\alpha^j \lambda_i + \nu + (1-\alpha^j)\lambda_0 \right)\bigg)\Bigg)^{n_i}. \qquad \end{eqnarray}
  \smallskip
\item [{\rm (c)}]  If  $\ell$ divides $k$, then there are orbit representatives $\lambda_i, \ i \in \tt I$, for the action of $\sigma_{\PP}$ on the roots of $h$ so that  $h = \gamma\prod_{i \in \tt I} h_i^{n_i}$ for some $\gamma \in \FF^*$ and integers $n_i \geq 1$,  where
\begin{equation}\label{eq:kp=ql} h_i(x) = \prod_{(\zeta,\varepsilon) \in \PP}  \big(  x - \sigma_{\zeta,\varepsilon}(\lambda_i)\big)   
= \Bigg(\prod_{\nu \in \G}  \prod_{j=0}^{\ell-1} \bigg(x-\left(\alpha^j \lambda_i + \nu + (1-\alpha^j)\textstyle{\frac{\beta}{1-\alpha}} \right)\bigg)\Bigg)^{n_i}. \end{equation}
\end{enumerate} 
If $\chara(\FF) = 0$, then $\G = \{0\}$, and  $\tau_{1,\G} = \{\mathsf{id}_{\A_h}\}$.  
\end{cor}
 
\begin{proof}   We may assume $k \geq 2$, since the first case follows directly from Theorem \ref{T:Pinfo}. 

Recall that $\G$ is a finite subgroup of $(\FF,+)$ and $\G = \{0\}$ when $\chara(\FF) = 0$ by (1) of 
Lemma \ref{L:PandG}.  Thus,  whenever $\G \neq \{0\}$,  we can suppose  $\chara(\FF) = p > 0$.  

Now if (a) holds, then either $\G = \{0\}$ and $\aut(\A_h) \cong  \FF[x]$, or else  $\G = \FF_p \nu_1 + \cdots + \FF_p \nu_d$ for some $d$.  Assume  $\lambda_i, i \in {\tt I}$, are  roots of $h$ in $\overline \FF$, which are representatives for the orbits of roots of $h$ in $\overline \FF$  under the affine bijections $\sigma_{1,\nu}$ for $\nu \in \G$.    Since each orbit is of size $p^d$, we have $k = q p^d$.  Then $h$ has the form displayed in (a).   When $\G = \{0\}$,  then $\aut(\A_h) \cong \FF[x]$,  $\lambda_i, i \in \tt I$, are the distinct roots of
$h$ in $\overline \FF$,  and $k = |\tt I |$ in this case.

Now suppose that  $\aut(\A_h) \not \cong \FF[x]  \rtimes \tau_{1,\G}$.  By Theorem \ref{T:Pinfo},
$\aut(\A_h)  \cong   (\FF[x]  \rtimes \tau_{1,\G}) \rtimes  \langle \tau_{\alpha,\beta} \rangle$,
where $\alpha$ is primitive $\ell$th root of unity for some $\ell > 1$ that divides $k$ or $k-1$. 
  
When $\ell$ divides $k-1$,  then as we have seen previously, there is one orbit of size
one under the action of $\sigma_{\alpha,\beta}$ generated by the root $\lambda_0: = \beta/(1-\alpha) \in \FF$.
Either the group $\G = \{0\}$, or $\chara(\FF) = p > 0$ and $\G$ has order $p^d$ for some $d \geq 1$, and $\G$ is invariant under multiplication by the cyclic group generated by $\alpha$ by (3) of
Lemma \ref{L:PandG}.   Under this action of the group  $\langle \alpha \rangle$,  there is one orbit of size 1 (namely $\{0\}$), and all the other orbits have size $\ell$.   Thus,  $r \ell + 1 = p^d$ for some $r \geq 0$.  

Consider the orbits of roots under the group generated by the maps $\sigma_{\alpha,\beta}$
and $\sigma_{1,\nu}$ as $\nu$ ranges over the elements of $\G$.   One such orbit
is $\{ \lambda_0 + \nu \mid \nu \in \G\}$.    Assume $\lambda_i$ for $i\in {\tt I}$ are
the representatives for the other orbits.      Then $h$ has the factorization into
linear factors given in \eqref{eq:kpminus1} for some $\gamma \in \FF^*$,
and $n_i \geq 1$.     Counting roots of $h$ in $\overline \FF$, we have
$q\ell + 1 = k$ when $\G = \{0\}$, and 
$q\ell p^d + p^d = ( r\ell+1)(q\ell +1) = \ell(r+q+rq\ell) + 1 = k$, when 
$\G \neq \{0\}$ and $ \chara(\FF) = p > 0$.   

The case when $\ell$ divides $k$ is similar and follows the same line of reasoning - 
just omit the factors of $h$ involving $\lambda_0$, and use the
fact that $\sigma_{\alpha,\beta}^j(\lambda_i+\nu) 
=\alpha^j(\lambda_i+\nu) + (1+\alpha+ \cdots + \alpha^{j-1})\beta$.   In this case, counting
roots gives either $q \ell = k$ ($\G = \{0\}$) or   $q p^d \ell = q(r\ell+1)\ell = k$ ($\G \neq \{0\}$,
$\chara(\FF) = p > 0$).   \end{proof} \vspace{-.1 truecm} 

\begin{remark}  Suppose $\alpha \in \FF$ is an $\ell$th root of unity for $\ell > 1$.  Let $\G$ be
a finite subgroup of $(\FF,+)$ invariant under multiplication by $\alpha$
 (necessarily $\G = \{0\}$ when $\chara(\FF) = 0$).
  By choosing $\lambda_i$ for
$i$ in some index set  ${\tt I}$ so that $\lambda_0+\nu, \alpha^j( \lambda_i+\nu)+ \lambda_0(1-\alpha^j)$ are distinct
for $\nu \in \G$, $i \in \tt I \cup \{0\}$,  and $j = 0,1,\dots,\ell-1$, 
and taking arbitrary $n_i \geq 1$ for $i \in \tt  I  \cup \{0\}$, we
can construct  $h$ as  in \eqref{eq:kpminus1} with  $\tau_{1,\G} \rtimes \langle \tau_{\alpha, \lambda_0(1-\alpha)}\rangle  \subset  \aut(\A_h)$.    Similarly,  if we choose  $\beta$ arbitrarily, $\G$ as
above,  and $\lambda_i$ for $i \in {\tt I}$ so that  $\alpha^j (\lambda_i+\nu)+ \beta(1-\alpha^j)/(1-\alpha)$
are all distinct for $\nu \in \G,\,  i \in {\tt I}$, and $j=0,1,\dots, \ell-1$,  and take arbitrary $n_i \geq 1$,   we can construct  $h$ as in \eqref{eq:kp=ql}
with $\tau_{1,\G} \rtimes \langle \tau_{\alpha, \beta}\rangle  \subset  \aut(\A_h)$. 
\end{remark}  

\begin{exam}  In this example,  we compute $\aut(\A_h)$ for any monic quadratic polynomial  $h(x) = x^2 - \zeta_1 x + \zeta_0 \in \FF[x]$.  Recall that  $(\alpha,\beta) \in \PP$  if and only if $h(\alpha x+\beta) = \alpha^{\degg h} h(x).$  Thus, 
\begin{eqnarray*} 
(\alpha,\beta) \in \PP &  \Longleftrightarrow  & (\alpha x +\beta)^2 - \zeta_1(\alpha x+\beta) + \zeta_0 = \alpha^2(x^2 - \zeta_1 x +\zeta_0) \\
& \Longleftrightarrow  & 2 \beta - \zeta_1 = -\alpha \zeta_1\ \hbox{\rm and} \ \beta^2 -\zeta_1 \beta + \zeta_0 = \alpha^2 \zeta_0 \\
& \Longleftrightarrow & \beta = \half  (1-\alpha)\zeta_1 \ \hbox{\rm and} \ \frac{1}{4}(1-\alpha)^2 \zeta_1^2 - \half (1-\alpha)\zeta_1^2+ (1-\alpha^2) \zeta_0 = 0.\end{eqnarray*} 
Therefore,  if $(\alpha,\beta) \in \PP$, then  either $(\alpha,\beta) = (1,0)$,  or $\alpha \neq 1$ and  $(1-\alpha)\zeta_1^2 - 2\zeta_1^2 + 4(1+\alpha)\zeta_0
= (1+\alpha)(4\zeta_0-\zeta_1^2) = 0$.   In the second event, either  $ \zeta_1^2 \neq 4\zeta_0$ and $(\alpha,\beta) = (-1,\zeta_1)$,   
or $\zeta_0 = \frac{1}{4} \zeta_1^2$ so that $h(x) = (x- \half \zeta_1)^2$.   We conclude that there are two possibilities:  either $\PP = \{(1,0), (-1,\zeta_1)\}$
which happens when $h(x)$ has two distinct roots, or $h(x) = (x- \half\zeta_1)^2$ and $\PP = \{(\alpha, (1-\alpha)\half\zeta_1)\}$.    In the first situation,  $\aut(\A_h) = \FF[x] \rtimes \langle \tau_{-1,\zeta_1}\rangle$ so that  $\aut(\A_h)/\FF[x]$ is a cyclic group of order two; in the second,  $\aut(\A_h) = \FF[x] \rtimes \FF^*$.   

In this calculation, we have tacitly assumed that $\chara(\FF) \neq 2$.   When $\chara(\FF) = 2$, then 
$(\alpha,\beta) \in \PP$ if and only if $\zeta_1 = \alpha \zeta_1$ and $\beta^2 - \zeta_1 \beta + \zeta_0 = \alpha^2 \zeta_0$.   Either $\zeta_1 \neq 0$  and  $\aut(\A_h) = \FF[x] \rtimes \tau_{\PP}$, where $\PP  = \{(1,0), (1,\zeta_1)\}$,
or else   $\zeta_1 = 0$ and $h(x) = x^2 + \zeta_0$.  
If $\zeta_0 = \lambda^2$ for some $\lambda \in \FF$, then $h(x) = (x+\lambda)^2$ and $(\alpha, (1-\alpha)\lambda) \in \PP$ for all $\alpha \in \FF^*$, so that $\aut(\A_h) = \FF[x] \rtimes \FF^*$.   If no such $\lambda$ exists, then  $\aut(\A_h) = \FF[x]$.   
\end{exam}

\begin{subsection}{The $\aut(\A_h)$ Invariants}  \end{subsection}

Throughout this section and the next,  we let $\fa = \aut(\A_h)$.   In this section, we determine  the invariants under $\fa$ in $\A_{h}$:
$$\A_{h}^{\fa} = \{ a \in \A_{h} \mid \omega(a) = a \quad \forall \  \omega \in \fa\}.$$

\begin{lemma}
For any $h\in\DD$, $\A_{h}^{\fa}=\DD^{\fa}=\DD^{\PP} =\{r \in \DD \mid \tau_{\zeta,\varepsilon}(r) = r \ \forall \ (\zeta,\varepsilon) \in \PP\}$.
\end{lemma} 
\begin{proof}
Let $\FF[x]\subseteq \fa$ be the subgroup of automorphisms of $\A_{h}$ of the form $\phi_{r}$, for $r\in\FF[x]$. We will first  show that $\DD = \A_{h}^{\FF[x]}$. The inclusion $\DD\subseteq\A_{h}^{\FF[x]}$ is clear, since $\phi_r(x) = x$ for all $r \in \DD$. We will prove that the reverse inclusion holds as well.

Assume by contradiction that there is  $a\in\A_{h}^{\FF[x]}\setminus\DD$, say $a=\sum_{i=0}^{m}f_{i} \hat y^{i}$ with $f_i = f_{i}(x)\in\DD$, $m\geq 1$, and $f_{m}\neq 0$. We can further assume $f_{0}=0$, so $a=\sum_{i=1}^{m}f_{i}\hat y^{i}$. Take $g\in\DD\cap\centh$. Then 

$$
0=\phi_{g}(a)-a=\sum_{i=1}^{m}f_{i}\left((\hat y + g)^{i}-\hat y^{i}\right).
$$
For $0\leq k\leq m-1$, the coefficient of $\hat y^{k}$ in the sum above is $\sum_{i=1}^{m}c_{i, k}f_{i}g^{i-k}$, where $c_{i, k}={i\choose k}$ if $k<i$ and $c_{i, k}=0$ otherwise.

Assume first that $\chara(\FF)=0$. Take $g=1$ and $k=m-1$ above. Then we get $m f_{m}=0$, which is a contradiction. Now suppose $\chara(\FF)=p>0$,  and take $g=x^{np}$, where $n$ is chosen so that $np>\max \{ \degg f_{i} \mid 1\leq i\leq m\}$, and $k=0$. We have $\sum_{i=1}^{m}f_{i}g^{i}=0$. For every $i$, either $f_{i}=0$ or
$$
inp\leq \degg f_{i}g^{i}< (i+1) np.
$$
This implies that $f_{m}g^{m}=0$, so $f_{m}=0$, which is a contradiction.  Thus $\A_{h}^{\FF[x]}\subseteq\DD$, and equality is proved.

The above shows that $\A_h^\fa \subseteq \DD^{\fa} \subseteq \DD^{\PP}$.   However, since $\phi_r(x) = x$ for all $r \in \DD$,   $\DD^{\fa} = \DD^{\PP}$, and the rest follows. 
\end{proof}

 Next we determine the invariants under $\fa$ in $\DD$:
$$\DD^{\fa} = \{ r \in \DD \mid \omega(r) = r \quad \forall \  \omega \in \fa\} = 
 \DD^{\PP} = \{r \in \DD \mid  \tau_{\zeta,\varepsilon}(r) = r \ \forall \  (\zeta,\varepsilon) \in \PP\}.$$

\begin{lemma} \label{L:invgen}  Suppose $\DD^{\fa} \neq \FF$.   Then there  exists a unique monic polynomial $s$  of minimal degree in $\DD^{\fa}\setminus \FF$  with zero constant term  such that  $\DD^\fa = \FF[s]$.  \end{lemma}

\begin{proof}  
Let $s$ be a monic polynomial of minimal degree in  $\DD^{\PP}\setminus \FF$.   We may assume that $s$ has zero 
constant term.     Now for every $r \in \DD^{\PP}$,   $r = sf + g$ for some $f,g \in \DD$ with $\degg g < \degg s$.
Applying $\tau_{\zeta,\varepsilon}$ to that relation gives
$$r = s \tau_{\zeta,\varepsilon}(f) + \tau_{\zeta,\varepsilon}(g),$$
and subtracting that from the above gives  $0 = s (f -\tau_{\zeta,\varepsilon}(f)) + g - \tau_{\zeta,\varepsilon}(g)$.
Since this is true for all $(\zeta,\varepsilon) \in \PP$, and since  $\tau_{\zeta,\varepsilon}$ preserves degree,   we have that $f \in \DD^{\PP}$ and $g \in \FF$. Thus $\DD^{\PP} = s \DD^{\PP} \oplus \FF$.

Clearly $\FF[s] \subseteq \DD^\fa=\DD^{\PP}$.  For the other direction, we proceed by induction on the
degree of an element of $\DD^{\fa}$; the case of degree 0 being obvious.   Assuming the result for degree $< n$,
we suppose $r \in \DD^{\fa}$ has degree $n$ where $n \geq 1$.     Then there exist $f \in \DD^{\fa}$ and $\xi_r \in \FF$  such that
$r = sf+ \xi_r$.    By induction, $f \in \FF[s]$.   Hence so is $r$, and $\DD^{\fa} \subseteq \FF[s]$. The uniqueness
of such an $s$ is clear.    \end{proof}
 
\begin{thm}\label{T:autinv}  Suppose  $\fa = \aut(\A_h)$.    Then  
\begin{itemize}
\item [{\rm (i)}]  $\DD^{\fa} =  \DD$ \   if  $\fa = \FF[x]$,   and  $\DD^{\fa} = \FF$ if  \ $ \fa = \FF[x] \rtimes \FF^*$  and $|\FF| = \infty.$ 
\item[{\rm (ii)}]  $\DD^{\fa} =  \FF[t]$,  where the polynomial $t \in \DD$ can be taken as follows:
\begin{itemize}
\item[{\rm (a)}] If  $\tau_{\PP}=\tau_{1, \G}$, then  $t(x) = \prod_{\nu \in \G}  \left(x +\nu \right)$. 
\item[{\rm (b)}] If
$\tau_{\PP} = \tau_{1,\G} \rtimes \langle \tau_{\alpha,\beta} \rangle$, where $\alpha$ is a primitive $\ell$th root of unity for some $\ell > 1$,  then  $t(x) = \prod_{\nu \in \G}  \left(x + \frac{\beta}{\alpha-1}+\nu \right)^\ell$.
\end{itemize}
\end{itemize}  
\end{thm}

\begin{proof}   Assume $r \in \DD^{\fa}$ and $\degg r \geq 1$,
and let  $\Lambda$ be the set of roots of $r$ in $\overline \FF$.       Since every automorphism of the form $\phi_f$ leaves 
$\DD$ pointwise fixed,     the first part of (i) is clear.      We will assume we have nontrivial automorphisms in $\tau_{\PP}$.
For any  $\tau_{1,\nu} \in \tau_{1,\G}$, the equality $r(x+\nu) =\tau_{1,\nu}(r) =  r(x)$ implies that $\mu+\nu\in\Lambda$ for all
$\mu \in \Lambda$.    Thus $\G$ acts faithfully on $\Lambda$, and roots of $r$ in the same $\G$-orbit have the same multiplicity. This implies that $\degg r$ is divisible by $|\G|$.  

 In particular,  if $\tau_{\PP}=\tau_{1, \G}$,   then  we claim that the polynomial $s$ in Lemma \ref{L:invgen}
 is given by $s(x) = t(x) - t(0)$,   where $t(x) = \prod_{\nu \in \G}  \left(x +\nu \right)$.      Indeed,  it is easy to see  that the polynomial $t$ belongs to $\DD^{\fa}$ in case (a) of (ii).    Moreover, $t(x)-t(0)$ is a monic polynomial of degree $| \G |$ in $\DD^{\fa}$  with zero constant term.   Since every $r \in \DD^{\fa}\setminus \FF$ has $\degg r \geq |\G|$,  $t(x)-t(0)$ is the polynomial $s$ in Lemma \ref{L:invgen}.   Finally, $\FF[t] = \FF[s] = \DD^{\fa}$ to give (ii)(a).       
 
In all the remaining possibilities for $\fa = \aut(\A_h)$, coming from   Theorem~\ref{T:Pinfo}, there exists an automorphism of the form  $\tau_{\alpha,\beta}$, with $(\alpha,\beta) \in \PP$ and $\alpha \neq 1$.  
Since $\degg r \ge 1$, it follows from considering the leading coefficient of $r = \tau_{\alpha, \beta}(r)$ that $\alpha^{\degg r} = 1$, and thus  when $r \not\in \FF$, $\degg r$ is at least the multiplicative order of any $\alpha \in \FF^*$ with $(\alpha, \beta) \in \PP$ for some $\beta \in \FF$.

Now when $\fa = \FF[x] \rtimes \FF^*$ in Theorem~\ref{T:Pinfo},   $\FF^*$ is identified with $\tau_{\PP} = \{\tau_{\alpha, (1-\alpha)\lambda} \mid \alpha \in \FF^*\}$, where $\lambda\in\FF$ is the unique root of $h$.     If $r \in \DD^{\fa}$ with $\degg r \ge 1$, then by the previous paragraph $\degg r$ is greater than or equal to the multiplicative order of every $\alpha \in \FF^*$.  If $\FF$ is infinite, there is no upper bound on the order of elements of $\FF^*$, so no such $r$ can exist.  Hence, we have the second part of (i).      

Assume now $\tau_{\PP} = \tau_{1,\G} \rtimes \langle \tau_{\alpha,\beta} \rangle$, where $\alpha$ is a primitive $\ell$th root of unity for some $\ell > 1$.  It can be further assumed that $\frac{\beta}{1-\alpha}$ is not a root of $r$ (if necessary, replace $r$ by $r+1$). Recall from the proof of Theorem~\ref{T:autosA_hprod} that $\tau_{\alpha, \beta}^i = \tau_{\alpha^i, \frac{1-\alpha^{i}}{1-\alpha}\beta}$ for all $i \geq 0$, so $|\langle \tau_{\alpha,\beta} \rangle|=\ell$. Since  $r \in \DD^\fa$, we have $r(x)=r(\alpha x+\beta)$ and $\alpha \mu+\beta\in\Lambda$ for all $\mu\in\Lambda$. Thus,  we have an action of $\langle \tau_{\alpha,\beta} \rangle$ on $\Lambda$, defined by $\tau_{\alpha, \beta}^i\ .\ \mu := \alpha^i\mu + \frac{1-\alpha^{i}}{1-\alpha}\beta$. Given our assumption that $\frac{\beta}{1-\alpha}\notin\Lambda$, this is a faithful action. Furthermore, the multiplicity is constant within each $\G$-orbit. The above shows that $\degg r$ is divisible by $\ell$.

Finally, note that $| \G|$ and $\ell = | \tau_\PP / \tau_{1, \G} |$ are coprime by Remark \ref{R:coprime}. Therefore, in case (ii)(b) the degree of the polynomial $r$ is divisible by the coprime integers $| \G|$ and $\ell$, so $\degg r\geq \ell | \G|$. Observe that 
\begin{eqnarray*} \tau_{\alpha,\beta}\left(x + \textstyle{\frac{\beta}{\alpha-1}} + \nu\right)  &=&  \alpha x + \beta + \textstyle{\frac{\beta}{\alpha-1}} + \nu \\
&=&  \alpha x + \textstyle{\frac{\alpha \beta}{\alpha-1}} + \nu = \alpha\left (x+ \textstyle{\frac{\beta}{\alpha-1}} +\alpha^{-1}\nu\right).
\end{eqnarray*}
From Lemma~\ref{L:PandG}, we know that $\alpha \G = \G$,  hence $\alpha^{-1}\nu \in \G$.  Thus the polynomial  $t(x) = \prod_{\nu \in \G}  \left(x + \frac{\beta}{\alpha-1}+\nu \right)^\ell$  in (ii)(b) is invariant under the automorphisms in $\tau_{1,\G}$ and also under $\tau_{\alpha, \beta}$, so   $t(x)$ is invariant under $\fa$.  As above, since $\degg t= \ell | \G|$ and any non-constant $r\in \DD^{\fa}$ has $\degg r\geq \ell | \G|$,  we deduce that  $\DD^\fa = \FF[t]$. 
\end{proof}
 
\begin{subsection}{The Center of $\aut(\A_h)$} \end{subsection}
The explicit description of the automorphism group $\aut(\A_h)$ in Theorem \ref{T:autosA_hprod} enables us to determine the center of this group.   \medskip
 
\begin{prop}
Assume $\degg h\geq 1$. Then the center of $\fa = \aut(\A_h)$ is 
$$
\mathsf{Z}(\fa)=\big\{ \phi_{r} \, \mid\, r \in \DD_{\mathsf{Z}}\big\}\ \ \hbox{\rm where} \ \ \DD_{\mathsf{Z}} = \big \{r \in \DD \mid \tau_{\zeta, \varepsilon}(r)=\zeta^{\degg h - 1}r \quad \forall \ (\zeta, \varepsilon)\in\PP\big \}.
$$
In particular, $\FF h'$ is a subgroup of $\mathsf{Z}(\fa)$ (under our usual identification of $r \in \FF[x]$ with the automorphism $\phi_{r}$).\end{prop}

\begin{proof} We first argue that the centralizer of the normal subgroup  $\FF[x]$ in $\fa$ is $\FF[x]$ itself, so $\zfa$ is a subgroup of $\FF[x]$.  Take  $\omega \in\fa$ such that  $\omega^{-1} \circ \phi_f \circ \omega = \phi_f$ for all $f \in \FF[x]$, and write $\omega = \phi_r \circ \tau_{\zeta, \varepsilon} \in \aut(\A_h) = \FF[x] \rtimes \tau_{\mathbb P}$.  Then by (\ref{eq:autnorm}), 
$$\phi_f = \omega^{-1} \circ \phi_f \circ \omega = \tau_{\zeta, \varepsilon}^{-1} \circ \phi_r^{-1} \circ \phi_f \circ \phi_r \circ \tau_{\zeta, \varepsilon} = \tau_{\zeta, \varepsilon}^{-1} \circ \phi_f \circ \tau_{\zeta, \varepsilon} = \phi_{\tilde f},$$
where $\tilde f (x) = \zeta^{\degg h -1} f (\zeta^{-1} ( x - \varepsilon))$.  This implies that $f( \zeta x + \varepsilon) = \zeta^{\degg h - 1} f(x)$ for all $f \in \FF[x]$.  Setting $f = h$ gives $\zeta^{\degg h - 1} h=h(\zeta x+\varepsilon)=\zeta^{\degg h} h$, which implies $\zeta=1$. Now set $f(x)=x$ to get $x+\varepsilon=x$, so $\varepsilon=0$. It follows that $\psi =\phi_{r}\in\FF[x]$.  This shows that the centralizer  $\mathsf{C}_{\fa}(\FF[x]) \subseteq \FF[x]$, and the other containment is trivial, so we have equality.  
 
Now  $\omega = \phi_r  \in\zfa$ if and only if  $\phi_{r}$ commutes with $\tau_{\zeta, \varepsilon}$, for every $(\zeta, \varepsilon)\in\PP$.  Equation (\ref{eq:autnorm}) gives that $\tau_{\zeta, \varepsilon}^{-1} \circ \phi_r \circ \tau_{\zeta, \varepsilon} = \phi_{\tilde r}$, 
where $\tilde r(x) = \zeta^{\degg h - 1} r( \zeta^{-1}(x - \varepsilon))$.  Thus the condition that $\phi_r = \tau_{\zeta, \varepsilon}^{-1} \circ \phi_r \circ \tau_{\zeta, \varepsilon}$ is equivalent to the condition that $r ( \zeta x + \varepsilon) = \zeta^{\degg h - 1} r(x)$, from which follows the desired result, 
$$\mathsf{Z}(\fa)=\big\{ \phi_{r} \, \mid\, r \in \DD_{\mathsf Z}\big\}, \ \ \hbox{\rm where} \ \  \DD_{\mathsf Z} = \big\{r \in \DD \mid \tau_{\zeta, \ve}(r)=\zeta^{\degg h - 1}r \quad \forall \ (\zeta, \ve)\in\PP\big\}.$$

Let $(\zeta, \ve)\in\PP$. Then, by definition, $h(\zeta x +\ve) = \zeta^{\degg h} h(x)$.  Taking the derivative of both sides shows that  $\zeta h'(\zeta x +\ve) = \zeta^{\degg h} h'(x)$, so $h'(\zeta x +\ve) = \zeta^{\degg h -1} h'(x)$. If we multiply both sides of this equation by an arbitrary $\lambda \in\FF$,  we see  that  $\FF h' \subseteq \DD_{\mathsf Z}$.   Under  our identification of $\{ \phi_f \mid f \in \FF[x] \}$ with $\FF[x]$, we have $\FF h' \subseteq \mathsf{Z}(\fa)$,  and  
$\FF h'$ is clearly a subgroup under addition.
\end{proof}

\begin{lemma}\label{L:qgen}
Assume $\degg h \ge 1$ and $\DD_{\mathsf Z}  \neq \{ 0 \}$, where $\DD_{\mathsf Z} =  \big\{ r \in \DD \mid \tau_{\zeta, \ve}(r)=\zeta^{\degg h - 1}r \ \forall \ (\zeta, \ve)\in\PP\big\}$.  Suppose $q \neq 0$ is the monic polynomial in $\DD = \FF[x]$ of minimal degree such that $q \in \DD_{\mathsf Z}$.   Then   $\DD_{\mathsf Z} = q \DD^{\fa}$.
\end{lemma}   
\begin{proof}
If $f = q r$, where $r \in \DD^{\fa}$, then for all  $(\zeta, \varepsilon) \in \PP$, $\tau_{\zeta,\varepsilon}(r)  = 
r$,    and we have   $\tau_{\zeta, \ve} (f) = \tau_{\zeta, \ve} (q) \tau_{\zeta, \ve} (r) = \zeta^{\degg h-1}qr = \zeta^{\degg h-1} f$, so $f \in \DD_{\mathsf Z}$.

For the other containment, assume $f \in \DD_{\mathsf Z}$, and use the division algorithm to write $f =qr + g$ with $r,g \in \FF[x]$ and $\degg g < \degg q$.  Then for  $(\zeta, \ve) \in \PP$,  we have  
$$\tau_{\zeta,\ve}(f) = \zeta^{\degg h - 1} f = \zeta^{\degg h - 1}q \tau_{\zeta, \ve}(r) + \tau_{\zeta, \ve}(g),$$
so that 
$f = q  \tau_{\zeta, \ve}(r) + \zeta^{- \degg h + 1} \tau_{\zeta, \ve}(g)$.  Subtracting $f = qr+g$ from this expression gives
$0 = q\big(\tau_{\zeta,\ve}(r)-r\big) + \zeta^{- \degg h + 1} \tau_{\zeta, \ve}(g)-g$.   
Since $\degg \tau_{\zeta, \ve}(g) = \degg g < \degg q$, this forces $\tau_{\zeta,\ve}(r) = r$, that is $r \in \DD^{\fa}$,  and $g = 0$ by the minimality of $\degg q$.     Thus, we have $f \in q \DD^{\fa}$.
\end{proof}
Combining these results with the description of the invariants $\DD^{\fa}$ in Theorem \ref{T:autinv}, we obtain the main result of this section --  a description of 
the center of $\aut(\A_h)$.  \medskip

\begin{thm}\label{T:centaut} Assume $\degg h\geq 1$. Let $\fa = \aut(\A_h)$, the automorphism group of $\A_h$.   The center $\mathsf{Z}(\fa)$ of 
$\fa$  is $\mathsf{Z}(\fa) = \{ \phi_r  \mid  r \in \DD_{\mathsf Z}\}$,   where
 $\DD_{\mathsf Z}  =  \{r \in \DD \mid r(\zeta x+ \varepsilon) = \zeta^{\degg h-1} r(x) \ \forall \ (\zeta,\varepsilon) \in \PP\}$, and
 $\zfa$ and $\DD_{\mathsf Z}$ are as follows:
\begin{itemize}  
\item[{\rm (1)}]   If  $\fa = \FF[x]$, then  $\DD_{\mathsf Z} = \DD$ and $\mathsf{Z}(\fa) = \FF[x] = \fa$.
\item[{\rm(2)}]   If  $\fa = \FF[x] \rtimes \tau_{1,\G}$,  then $\DD_{\mathsf Z} =   \DD^{\fa} = \FF[t]$ where
$t(x) = \prod_{\nu \in \G} (x+\nu)$.       Hence $\mathsf{Z}(\fa) = \{\phi_r \mid  r \in \FF[t]\}$.  
\item[{\rm(3)}]  If $\fa = \FF[x] \rtimes \FF^*$ and  $|\FF| = \infty$,  then $h = \gamma(x-\lambda)^n$ for some
$\gamma \in \FF^*$ and some $\lambda \in \FF$,  and $\DD_{\mathsf Z} = (x-\lambda)^{n-1}\DD^{\fa} = \FF(x-\lambda)^{n-1}$.
Hence $\mathsf{Z}(\fa) = \{\phi_r \mid  r \in \FF(x-\lambda)^{n-1}\}$.   
\item[{\rm (4)}]  If  $\fa = \FF[x] \rtimes \tau_{\PP}$,  where  $\tau_\PP = \tau_{1,\G} \rtimes \langle \tau_{\alpha,\beta} \rangle$ and
 $\alpha$ is a primitive $\ell$th  root of unity  for some $\ell>1$, then $\DD_{\mathsf Z} = q \FF[t]$,   where
$$q(x) =  \prod_{\nu \in \G}  \left (x + \textstyle{\frac{\beta}{\alpha-1}} + \nu\right)^n, \qquad t(x) = \prod_{\nu \in \G}\left(x + \textstyle{\frac{\beta}{\alpha-1}} + \nu\right)^\ell$$ and $0 \leq n < \ell$ is such that $n |\G| \equiv \degg h-1 \modd \ell$.    Hence,
$\mathsf{Z}(\fa) = \{\phi_r \mid  r \in q \FF[t]\}$.  
\end{itemize}    
\end{thm}
 
\begin{proof} It will be seen in the course of the proof that in all cases $\DD_{\mathsf Z}\neq \{ 0\}$, so from Lemma \ref{L:qgen},  we know that $\DD_{\mathsf Z} = q \DD^{\fa}$,  where 
$q$ is the  monic polynomial of minimal degree in $\DD_{\mathsf Z}$.    Since we have determined 
$\DD^{\fa}$ in Theorem \ref{T:autinv},  we need to find the polynomial $q$.      
For all $(\zeta,\ve) \in \PP$
we have from $q(\zeta x+\ve) = \zeta^{\degg h -1}q(x)$ that
$\zeta^{\degg q} = \zeta^{\degg h - 1}$. 

Let's consider the various cases arising from Theorem \ref{T:Pinfo}  and Corollary \ref{C:Pinfo}:
\begin{itemize}
\item [{\rm (i)}]   If $\fa = \FF[x]$ or $\fa = \FF[x] \rtimes \tau_{1,\G}$,  then $\DD_{\mathsf Z} = \DD^{\fa} = \A_h^{\fa}$ (and $q = 1$). 
\item [{\rm (ii)}]   If  $\fa  = \FF[x] \rtimes \FF^*$,  where $|\FF| = \infty$ and $\FF^*$ is identified with the group
$\{\tau_{\alpha,(1-\alpha)\lambda} \mid \alpha \in \FF^*\}$, then by the above, 
 $\alpha^{\degg q} = \alpha^{\degg h - 1}$ for all $\alpha \in \FF^*$,   which  forces $\degg q = \degg h -1$.   Recall that this case
occurs when $h(x) = \gamma (x-\lambda)^n$ for some $\gamma\in\FF^*$,  $\lambda \in \FF$, and $n \geq 1$.    The  monic polynomial $(x-\lambda)^{n-1}$ has degree equal to $\degg h-1$, and it is in $\DD_{\mathsf Z}$.
Thus, $q(x) = (x-\lambda)^{n-1}$,  and $\DD_{\mathsf Z} = (x-\lambda)^{n-1} \DD^{\fa}$. 
\item [{\rm (iii)}]   In all the remaining cases,  the group $\tau_{\PP}$ is finite.   We may assume $| \tau_{\PP}/\tau_{1,\G}| = \ell > 1$ or
else we are in case (2).      Write $\tau_{\PP} = \tau_{1,\G} \rtimes \langle \tau_{\alpha,\beta} \rangle$ where $\alpha$ is a primitive
$\ell$th root of 1.    Note that $\ell$ and $|\G|$ are coprime by Remark \ref{R:coprime}.  

We have shown that $\DD^\fa = \FF[t]$ where $t(x) = \prod_{\nu \in \G} \left(x + \frac{\beta}{\alpha-1} + \nu\right)^\ell$.   
Since $| \G |$ is invertible mod $\ell$ we can find $n$ so $0 \leq n < \ell$ and $n |\G| \equiv  \degg h-1 \modd \ell$.  
Set  $u(x) = \prod_{\nu \in \G} \left (x + \frac{\beta}{\alpha-1} + \nu\right)^n$.  Now $u(x + \xi) = u(x)$ for all $\xi \in \G$,
and $u(\alpha x+\beta) = \alpha^{n|\G|}u(x) = \alpha^{\degg h-1} u(x)$.   These expressions show that
$u \in \DD_{\mathsf Z}$.       Hence, there exists a polynomial $f(t) \in \FF[t]$ so that $u = q f(t)$.     However, since
the degree of $t$  in $x$ is  $\ell |\G|$ and the degree of $u$ in $x$ is $n|\G|$ and $n < \ell$, it must be that $f(t) \in \FF$. 
But since both $q$ and $u$ are monic, this says $q = u$.\end{itemize}  \end{proof}

\begin{exam}  Assume $h(x) = x^n$ for some $n \geq 1$.    Then by Theorem \ref{T:Pinfo},   $\fa = \aut(\A_h) = \FF[x] \rtimes \FF^*$, where $\FF^*$ is identified with the automorphisms $\{\tau_{\alpha,0}  \mid  \alpha \in \FF^*\}$.  If $\FF$ is infinite, the monic polynomial generator of $\DD_{\mathsf Z}$ is  $q(x) = x^{n-1}$ by Theorem \ref{T:centaut}, and according to Theorem \ref{T:autinv},  the invariants are given by $\DD^{\fa} = \FF$. Thus, in this case $\DD_{\mathsf Z} = \FF x^{n-1}$  and  $\mathsf{Z}(\fa) = \{\phi_f \mid  f \in \FF x^{n-1}\}$.   If $| \FF^*| = \ell < \infty$, then part (4) of Theorem \ref{T:centaut} shows that the monic polynomial generator of $\DD_{\mathsf Z}$ is $q(x) = x^{m}$, where $0 \le m < \ell$ and $m \equiv n-1 \modd \ell$.  Now Theorem \ref{T:autinv} asserts that $\DD^{\fa} = \FF[t]$,  where $t(x) = x^\ell$,  thus $\DD_{\mathsf Z} = x^{m} \FF[x^\ell]$ and $\mathsf{Z}(\fa) = \{\phi_f \mid  f \in x^{m}\FF[x^\ell]\}$.
\end{exam}
 
\begin{remark}  In the case of the Weyl algebra,  the center of $\aut(\A_{1})$ is trivial by  \cite[Prop.~3]{KA11}.    However, when $h \not \in \FF^*$, we can have the opposite extreme.  For example, if  $h=x^{2}(x-1)$, then $\PP=\{(1, 0)\}$,  as any permutation of the roots of $h$ has to fix $0$ and $1$ (since they have different multiplicities),  and the affine permutations determined by elements of $\PP$ can have at most $1$ fixed point, except for the identity map.   So $\aut (\A_{h})=\FF[x]$ is commutative,  and its center is the entire automorphism group in this case. 
\end{remark}

\begin{subsection}{Automorphisms of the Weyl Algebra}\label{sec:automWeyl}\end{subsection} 

In this section we contrast the previous results on automorphisms of $\A_h$ for $h \not\in \FF$, with known results on the automorphisms of the Weyl algebra $\A_1$.  The Weyl algebra has more automorphisms because of its high degree of symmetry.

Let $\mathsf{SL_2(\FF)}$ denote the special linear group of $2 \times 2$ matrices over $\FF$ of
determinant 1.   Each matrix  $\tt{S} =  \left( \begin{smallmatrix} \alpha & \gamma \\  \beta & \varepsilon \end{smallmatrix} \right)  \in \mathsf{SL_2(\FF)}$ determines an automorphism $\varphi_{\tt S}$  of $\A_1$ given by 

\begin{equation}   x  \mapsto  \alpha x + \beta y,   \qquad    y \mapsto   \gamma x + \varepsilon y. \end{equation}

The matrix $\tt T:= \left( \begin{smallmatrix} 0 & 1 \\  -1 & 0 \end{smallmatrix} \right) \in \mathsf{SL_2(\FF)}$
corresponds to the automorphism $\tau := \varphi_{\tt T}$  of $\A_1$ given by  $x \mapsto -y$, \, $y \mapsto x$. 
And $\tau^{-1}$ corresponds to the automorphism with $x \mapsto y$, \, $y \mapsto -x$.   Note that $\tau^2 = \mathsf{-I}$, \, $\tau^4 = \mathsf{I}$, \, and  $\tau^3 = \tau^{-1} = \left( \begin{smallmatrix} 0 & -1 \\  1 & \ 0 \end{smallmatrix} \right)$.    
\medskip 

For each $f \in \FF[x]$,  there is an automorphism $\phi_f$ with $\phi_f(x) = x$ and
$\phi_f(y) = y+f$, just as for the algebras $\A_h$.   However, in the $\A_1$ case, observe that
\begin{eqnarray*} \left(\tau^{-1} \circ  \phi_{-f} \circ \tau \right)(x) &=& x + f(y) \\
\left(\tau^{-1} \circ  \phi_{-f} \circ \tau\right)(y) &=&  y. \end{eqnarray*}
Hence,  the automorphisms  $\psi_f := \tau^{-1} \circ  \phi_{-f} \circ \tau$ for $f \in \FF[x]$ give the analogues of the maps $\phi_f$ but with the roles of $x$ and $y$ interchanged.  \medskip

\begin{remark}  Unlike the situation for $\A_h$, with $\degg h \ge 1$, 
the subgroup $\FF[x]$ fails to be normal in $\aut(\A_1)$, which can be seen
from the above calculation.   \end{remark} 

The following provide generating sets of automorphisms for $\aut(\A_1)$.  (Compare 
\cite {MakLim84} and \cite{stafford87},  and see also \cite{KA11} for part (iii).)

\begin{thm}\label{T:autgens}   Each of the following sets  gives a generating set for the automorphism
group $\aut(\A_1)$:
\begin{itemize} 
\item[{\rm (i)}] $\{\phi_f  \mid f \in \FF[x]\} \cup \{\psi_f \mid f \in \FF[x]\}$,
\item[{\rm (ii)}]  $\{\varphi_{\tt S} \mid {\tt S} \in \mathsf{SL}_2(\FF) \}  \cup
\{\phi_f  \mid f \in \FF[x]\}$,
\item[{\rm (iii)}] $\{\tau, \phi_f  \mid f \in \FF[x]\}$,
\item[{\rm (iv)}] $\{\tau, \psi_f  \mid f \in \FF[x]\}$.
\end{itemize}
\end{thm} 

\begin{subsection}{Dixmier's Conjecture}\end{subsection}  In \cite[Problem 1]{Dix68}, Dixmier asked if every algebra endomorphism of  the $n$th Weyl algebra must be an automorphism when $\chara(\FF) = 0$.   This conjecture was shown to be equivalent to the longstanding Jacobian conjecture (see \cite{Tsu} and \cite{BK}).   In this section, we explore whether monomorphisms for the algebra $\A_h$ with $\degg h \geq 1$ necessarily are automorphisms.   \medskip

\begin{prop}
Assume $h=x^{n}$ for some $n\geq 1$,  and fix  $k\geq 1$. 
When  $\chara(\FF)=p>0$ assume further that $p$ does not divide $k$. 
Then there is an algebra monomorphism $\eta_{k} : \A_{h}\rightarrow \A_{h}$ such that $\eta_{k}(x)=x^{k}$ and $\eta_{k}(\hat y)=\frac{1}{k}x^{(k-1)(n-1)}\hat y$.  If $k\geq 2$,  then $\eta_{k}$ is not an automorphism.
\end{prop}

\begin{proof} Note that 
\begin{align*}
[\eta_{k}(\hat y), \eta_{k}(x)] &=\left[ \textstyle{\frac{1}{k}}x^{(k-1)(n-1)}\hat y, x^{k} \right] =  \textstyle{\frac{1}{k}}x^{(k-1)(n-1)} [ \hat y, x^{k} ]\\ &= \textstyle{ \frac{1}{k}} x^{(k-1)(n-1)} kx^{k-1+n}=x^{kn}=\eta_{k}(x^{n}),
\end{align*}
so there is an endomorphism $\eta_{k}$ as stated. This endomorphism is injective because 
\begin{align*}
\eta_{k}(x^{i}\hat y^{j}) &= \textstyle{\frac{1}{k^{j}}}x^{ik}\left(x^{(k-1)(n-1)}\hat y\right)^{j}\\
&= \textstyle{\frac{1}{k^{j}}}x^{j(k-1)(n-1)+ik}\, \hat y^{j} + \text{lower order terms in $\hat y$.}
\end{align*}
The above also shows that $\mathsf{im}(\eta_{k})\cap\DD=\eta_{k}(\DD)=\FF[x^{k}]$. If $k\geq 2$, then $x\notin\mathsf{im}(\eta_{k})$.   Thus $\eta_{k}$ fails to be surjective and consequently is not an automorphism. 
\end{proof}

When $\chara(\FF) = p > 0$, it is known (e.g.~Sec.~3.1 of \cite{KA11}) that Dixmier's conjecture fails to hold for $\A_1$.  The next result shows that the analogue of Dixmier's conjecture fails to hold for $\A_{h}$ for any $h$ with  $\degg h \ge 1$.

\begin{prop} Assume $\chara(\FF) = p > 0$ and  $\degg h \ge 1$.   Let $c \in \mathsf{C}_{\A_h}(x)= \FF[x,h^py^p]$.  Then there is an algebra monomorphism $\kappa_c : \A_h \to \A_h$ such that  $\kappa_c (\hat y) = \hat y + c$ and $\kappa_c (r) = r$ for all $r \in \FF[x]$.   If $c \not\in \FF[x]$, then $\kappa_c$ is not an automorphism of $\A_h$.
\end{prop}
\begin{proof}
Note that 
$$[ \kappa_c (\hat y), \kappa_c (x)] = [ \hat y + c , x] = [ \hat y, x] = h = \kappa_c (h),$$
so $\kappa_c : \A_h \to \A_h$ defines an algebra homomorphism.  That $\kappa_c$ is injective follows from the fact that $(\hat y + c)^i = \hat y^i + b$ for $b \in \bigoplus_{0 \le j < i} \DD \hat y^j$.  

Since $\kappa_c$ is an algebra monomorphism of $\A_h$, it follows that $\kappa_c \in \aut(\A_h)$ if and only if $\kappa_c$ is surjective.    If $\kappa_c \in \aut (\A_h)$, then by Theorem \ref{T:isos:A_h}, $\kappa_c (\hat y) \in \FF^* \hat y + \FF[x]$.  But since $\kappa_c (\hat y) = \hat y + c$, which is not in $\FF^* \hat y + \FF[x]$ whenever $c \not\in \FF[x]$,  it follows that $\kappa_c$ cannot be surjective if $c \not\in \FF[x]$. 
\end{proof}

\begin{subsection}{Restriction and Extension of Automorphisms}\end{subsection}

We assume here that there is an embedding   
of $\A_g$ into $\A_f$ where $f,g \in \FF[x]$.  We determine when an automorphism
of $\A_g$ extends to one of $\A_f$, and in the opposite direction, when an automorphism of $\A_f$ restricts to one of $\A_g$.  
 
\medskip
\begin{thm} Assume $\degg f \ge 0$, $\degg g\geq 1$, and $g = rf$.  Regard $\A_g = \langle x,\tilde y,1\rangle  \subseteq \A_f = \langle x,y,1\rangle$ with
$\tilde y = yr$.  
\begin{itemize}\item[{\rm (i)}] Suppose that $\omega  = \phi_q \circ \tau_{\alpha,\beta}  \in \aut (\A_g)$ so that  
$$\omega (x)=\alpha x+\beta, \quad \omega(\tilde y)=\alpha^{\degg g-1} \left(\tilde y+q(x)\right), \ \quad \text{and} \quad \alpha^{\degg g} g(x)=g(\alpha x+\beta),$$
as in Theorem \ref{T:autosA_hprod}.  Then $\omega \in \aut(\A_g)$ extends to an automorphism of $\A_f$ if and only if $\omega (f)=\alpha^{\degg f} f$ and $q$ is divisible by $r$. 
\item[{\rm (ii)}] Suppose that $\psi \in {\rm Aut}_\FF (\A_f)$.  Then $\psi$ restricts to an automorphism of $\A_g$ if and only if $\psi(g)=\lambda g$ for some $\lambda \in\FF^{*}$.
\end{itemize}  \end{thm}

\begin{proof}
(i)  Suppose that $\omega  = \phi_q \circ \tau_{\alpha,\beta}  \in \aut (\A_g)$  extends to an automorphism of $\A_f$.  Then since $\omega$ restricted to $\FF[x]$
is $\tau_{\alpha,\beta}$, it must be that $f( \alpha x + \beta) = \omega(f(x)) = \alpha^{\degg f} f(x)$
(compare Theorem \ref{T:isos:A_h}).   Applying  $\omega$ to the equation $g = rf$, we have 
$$\alpha^{\degg g} g = \omega (g) = \omega(rf) = \omega (r) \omega (f) = \omega (r) \alpha^{\degg f} f,$$
and therefore $\omega(r) = \alpha^{\degg g - \degg f} r$.     Moreover, 
\begin{equation}\label{eqn:phiExtends-yr}
\alpha^{\degg g- 1} (yr + q) = \omega (yr) = \omega (y) \omega (r) = \omega(y)(\alpha^{\degg g-\degg f}r).\end{equation}  Hence, $\omega(y) = \alpha^{\degg f - 1} y + s$ for some $s \in \DD$ and
$q = \alpha^{1 - \degg f} rs$,
so $r$ divides $q$.

Conversely, suppose that $\omega = \phi_q  \circ \tau_{\alpha,\beta} \in \aut(\A_g)$,  $\omega (f)=\alpha^{\degg f} f$, and $q$ is divisible by $r$. Write $q= rs$ for some $s \in \DD$.   Since $f(\alpha x+\beta) =\omega(f) = \alpha^{\degg f} f(x)$ and  $\omega(g) =g(\alpha x+\beta) = \alpha^{\degg g}g(x)$, it follows that $r(\alpha x+\beta) = \alpha^{\degg g-\degg f} r(x)$.  We claim
that $\omega$ agrees with the restriction of the  automorphism $\varphi = \phi_{s} \circ \tau_{\alpha,\beta} \in \aut(\A_f)$  to the subalgebra $\A_g$.
Indeed,  $\varphi(y) = \alpha^{\degg f-1} (y + s)$,  and  $\varphi(\tilde y) = \varphi(y)\varphi(r) = \alpha^{\degg f-1} (y + s)(\alpha^{\degg g-\degg f}r)= \alpha^{\degg g-1}(\tilde y + rs)=  \alpha^{\degg g-1}(\tilde y + q)= \omega(\tilde y)$.  Therefore,  $\varphi$ and $\omega$ agree on the generators $x,\tilde y$ of $\A_g$, and $\omega$ extends to the automorphism $\varphi$ of $\A_f$.

For (ii), assume $\psi \in \aut(\A_f)$.    If   $\psi$ restricts to an automorphism of $\A_g$, 
then by Theorem \ref{T:isos:A_h}, there is  $\alpha  \in \FF^*$ so that $\psi (g)=\alpha^{\degg g}g$.
Conversely, suppose that $\psi$ satisfies $\psi (g)=\lambda g$ for some $\lambda \in\FF^{*}$.  As $\degg g\geq 1$, it follows from $g(\psi (x))=\lambda g(x)$ that there are $\alpha \in \FF^*, \beta \in \FF$ with $\psi (x) = \alpha x + \beta \in \A_g$, and therefore $\psi^{-1}(x) = \alpha^{-1}(x - \beta) \in \A_g$.  Then it is easy to conclude that there exist $\mu\in\FF^*$ and $q \in \FF[x]$ so that $\psi (y)=\mu y+q$. If we apply $\psi$ to the defining relation of $\A_{f}$, we further deduce that $f(\alpha x+\beta)=\alpha\mu f(x)$, so in fact $\mu=\alpha^{\degg f - 1}$ and $f(\alpha x+\beta)=\alpha^{\degg f} f(x)$. Then 
$\lambda g(x) = g(\alpha x+\beta)$ implies that $\lambda=\alpha^{\degg g}$.   From this we deduce that $\psi (r(x))=r(\alpha x + \beta) = \alpha^{\degg g-\degg f} r(x)=\alpha^{\degg r} r(x)$.  It remains to prove that $\psi (\tilde y)\in \A_{g}$ and $\psi (\A_g) \supseteq \A_g$.  Observe that 
$$
\psi (\tilde y)=\psi(y)\psi(r)
= (\alpha^{\degg f-1} y+q)(\alpha^{\degg g-\degg f} r)  
=\alpha^{\degg g-1} \tilde y+\alpha^{\degg r}rq \in\A_{g}.
$$
Now if we let $s \in \FF[x]$ such that $s( \alpha x + \beta) = \alpha^{\degg r} rq$, it is straightforward to see that $\psi ( \alpha^{1-\degg g} (\tilde y - s)) = \tilde y$, and thus the image of the restriction of $\psi$ to $\A_g$ contains the generators $x$ and $\tilde y$.

\end{proof}

\begin{prop}\label{P:Hh}  For $0\neq h \in \FF[x]$, the subgroup $\mathsf{ H}_{h}=\{\omega  \in\aut(\A_1)\mid \omega(\A_h)=\A_{h}\}$  is normal  in $\aut(\A_1)$  if and only if $h\in\FF^{*}$. \end{prop}
 \begin{proof} That $\mathsf{H}_h$ is a subgroup is clear. Suppose $\omega \in\mathsf H_h$ is defined by $\omega (x)=x$ and $\omega(y)=y+x$.  Recall the automorphism $\tau\in \aut(\A_1)$ defined by $\tau (x)=-y$ and $\tau (y)=x$,  and observe that  $\tau\notin\mathsf H_h$.  Then 
$$
(\tau \circ \omega \circ \tau^{-1})(x)=\tau(y+x)=x-y.
$$
If $\mathsf H_h$ is normal in $\aut(\A_1)$,  then $\tau \circ \omega \circ \tau^{-1}$ restricts to an automorphism of $\A_{h}$, which is impossible unless $h \in \FF^*$,  since automorphisms of $\A_h$  must map $\FF[x]$ to itself when $h \notin \FF^*$. The converse is clear, as $\mathsf H_h=\aut(\A_h)$ if $h\in\FF^{*}$.
\end{proof}

\begin{section}{Relationship of the Algebras $\A_h$  to \\ Generalized Weyl Algebras} \end{section}

Given a ring $\mathsf{D}$, an automorphism $\sigma$ of $\mathsf{D}$, and a central element $a\in \mathsf{D}$, the {\it generalized Weyl algebra} $\mathsf{D}(\sigma, a)$ is the ring extension of $\mathsf{D}$ generated by $u$ and $d$, subject to the relations:
\begin{equation}\label{E:noeth:gwa1}
ub=\sigma (b)u, \quad\quad bd=d\sigma (b),  \quad\quad \text{for all $b\in \mathsf{D}$;}
\end{equation}
\begin{equation}\label{E:noeth:gwa2}
du=a, \quad\quad ud=\sigma (a).
\end{equation}
Generalized Weyl algebras were introduced by Bavula~\cite{bavula:gwar93}, who showed that  if $\mathsf{D}$ is a Noetherian $\FF$-algebra which is a domain, the automorphism $\sigma$ is $\FF$-linear,  and $a\neq 0$,  then $\mathsf{D}(\sigma, a)$ is a Noetherian domain.

\medskip
\begin{lemma}\label{L:OreGWA}{\rm [cf.~Lemma~\ref{lem:poly}]}
The following are generalized Weyl algebras over a polynomial ring $\mathsf{D}=\FF[t]$:
\begin{itemize}
\item[{\rm (i)}]   a  quantum plane
\item[{\rm (ii)}]   a quantum Weyl algebra
\item[{\rm (iii)}]   the polynomial algebra in two variables
\item[{\rm (iv)}]   the Weyl algebra.
\end{itemize}
\end{lemma}
\begin{proof}
Cases (i), (ii), and (iv) follow  from Examples 2, 4, and 1, respectively of \cite{BvO97}.   The remaining case can be seen by letting $\sigma$ be the identity automorphism of $\mathsf{D}$ and $a=t$, so that $\mathsf{D}(\sigma, a)\cong\FF[d, u]$.  
\end{proof}

In view of Lemma~\ref{lem:poly} and the preceding result, it is natural to inquire whether the algebras $\A_{h}$, for $h\notin\FF$, are generalized Weyl algebras. Theorem \ref{thm:A_h-not-GWApoly} gives an answer to this question (in the negative) when $\mathsf{D}$ is a polynomial ring in one variable.
 
\begin{lemma}\label{lem:princIdealGWA}
Assume $\mathsf D$ is a domain with $0\neq a \in \mathsf D$ $central$,  and let $\sigma : \mathsf D \to \mathsf D$ be an automorphism of $\mathsf D$.  If $a \not\in \mathsf D^\times$, then the only principal ideal of the generalized Weyl algebra $\mathsf D(\sigma, a)$ containing both $u$ and $d$ is $\mathsf D(\sigma, a)$.
\end{lemma}
\begin{proof}
Consider the natural $\mathbb{Z}$-grading on $\mathsf{D}(\sigma, a)$ where the elements of $\mathsf{D}$ have degree $0$, $d$ has degree $-1$ and $u$ has degree $1$.

Assume $v \mathsf{D}(\sigma, a)$ is a principal ideal of $\mathsf{D}(\sigma, a)$ generated by $v$ and containing $u$. Then, the equation $v b=u$, for $b\in \mathsf{D}(\sigma, a)$, implies that both $v$ and $b$ must be homogeneous with respect to the $\mathbb{Z}$-grading. Assume $v$ has degree $n<0$. Then we can write $v=cd^{-n}$ and $b=\tilde c u^{1-n}$, for some $c, \tilde c \in \mathsf D$. We have:
\begin{equation*}
u=(cd^{-n})(\tilde cu^{1-n})=(c\sigma^{n}(\tilde c)d^{-n}u^{-n})u.
\end{equation*}
The above equation implies that $du=a$ is a unit in $\mathsf D$, which is a contradiction. Hence, $v$ has degree $n\geq 0$. Similarly, assuming that $d \in v \mathsf{D}(\sigma, a)$, we conclude that $v$ has degree $n\leq 0$. It follows that if $v \mathsf{D}(\sigma, a)$ contains both $u$ and $d$, then $v \in \mathsf D$. But then the equation $v\tilde cu=u$, for $\tilde c \in \mathsf D$, implies that $v\mathsf{D}(\sigma, a)=\mathsf{D}(\sigma, a)$.
\end{proof}

\begin{thm}\label{thm:A_h-not-GWApoly}
Assume  $h \not \in\FF$. Then the algebra $\A_h$ is not a generalized Weyl algebra over a polynomial ring in one variable.
\end{thm}
\begin{proof}
Assume $h\neq 0$ and $\A_{h}\cong \mathsf{D}(\sigma, a)$, for $\mathsf{D}=\FF[t]$. First, notice that $a\notin\FF$, as otherwise we would have $ud=0=du$, and $\A_{h}$ would not be a domain, or else $u=d^{-1}$ and $\A_{h}$ would have nontrivial units. By~\cite[Prop.\ 2.1.1]{RS06} we need only consider three possibilities for $\sigma$:
\begin{itemize}
\item[(A)] $\sigma$ is the identity automorphism;
\item[(B)] $\sigma(t)=t-1$;
\item[(C)] $\sigma(t)=\xi t$, for some $\xi \in\FF^{*}$, with $\xi \neq 1$.
\end{itemize}

Notice that if $\sigma$ is the identity then $\mathsf{D}(\sigma, a)$ must be commutative and thus $h=0$, so case (A)  above does not occur.  Cases (B) and (C) are usually referred to as the {\it classical} and {\it quantum} cases, respectively.  

Let $\mathrm{Frac}( \A_h)$ be the skew field of fractions of $\A_{h}$. By Corollary~\ref{C:Weylfield}, $\mathrm{Frac}( \A_h)$ is the (first) Weyl field, i.e., the field of fractions of the Weyl algebra.  Thus, it follows by \cite[Prop.\ 2.1.1]{RS06} and \cite[Th\'e.\ 3.10]{AD94} that $\mathsf{D}(\sigma, a)$ must be of classical type, i.e., $\sigma(t)=t-1$.

Let the ideal $\BBh$ of $\A_{h}$ (resp.\ $\mathsf{J}$ of $\mathsf{D}(\sigma, a)$) be minimal with the property that $\A_{h}/\BBh$ (resp.\ $\mathsf{D}(\sigma, a)/\mathsf{J}$) is commutative. Then, by the defining relations of $\A_{h}$ and the fact that $h$ is normal, we have $\BBh=h\A_{h}$. In particular, $\BBh$ is a principal ideal,  and it follows that $\mathsf{J}$ is also principal.  In $\mathsf{D}(\sigma, a)$, the relations $u=[t, u]$ and $d=[d, t]$ show that $u, d\in \mathsf{J}$.  But Lemma \ref{lem:princIdealGWA} implies that $\mathsf{J} = \mathsf D(\sigma, a)$, and thus $h \A_h = \A_h$, so $h \in \FF^*$.
\end{proof}

\noindent  \textsc{Georgia Benkart} \\
\textit{\small Department of Mathematics, University of Wisconsin-Madison, Madison, WI 53706, USA}\\
\texttt{benkart@math.wisc.edu}\\ 

\noindent \textsc{Samuel A. Lopes} \\
\textit{\small CMUP, Faculdade de Ci\^encias, Universidade do Porto, 
Rua do Campo Alegre 687\\ 
4169-007 Porto, Portugal}\quad 
\texttt{slopes@fc.up.pt}\\

\noindent \textsc{Matthew Ondrus} \\
\textit{\small Mathematics Department,
Weber State University, 
Ogden, Utah, 84408 USA}\\
\texttt{mattondrus@weber.edu}

\end{document}